\documentclass[11 pt,reqno]{amsart}

\usepackage[utf8]{inputenc}
\title[Homogeneous Lagrangian submanifolds of $\Sl\times\Sl$]{Extrinsically homogeneous Lagrangian submanifolds of the pseudo-nearly Kähler $\Sl\times\Sl$}
\date{}
\author{Mateo Anarella}
\address{M. Anarella, Address 1: Université Polytechnique Hauts-de-France, Campus Mont Houy
59313 Valenciennes Cedex 9, France. Address 2: KU Leuven, Department of Mathematics, Celestijnenlaan 200 B – Box 2400, 3001 Leuven, Belgium}
\email{mateo.anarella@kuleuven.be}
\thanks{
    M. Anarella is supported by Methusalem grant METH/21/03 -- long term structural funding of the Flemish
    Government}

\usepackage{amsmath,amssymb,amsfonts,mathtools,amsthm}
\usepackage{amsfonts}
\usepackage{graphicx}
\usepackage{graphics}
\usepackage{amstext} 
\usepackage[italicdiff]{physics}
\usepackage{mathrsfs}

\usepackage[hidelinks,pdfencoding=auto,psdextra]{hyperref}
\usepackage{color,soul}
\usepackage[dvipsnames]{xcolor}

\usepackage{float}
\usepackage{tabularray}
\usepackage{tasks}
\usepackage{tikz-cd}
\usepackage{tikz}
\usetikzlibrary{graphs,quotes}

\usepackage{geometry}
\setul{}{0.3 ex}

\newtheorem{theorem}{Theorem}

\newtheorem{lemma}[theorem]{Lemma}

\newtheorem{proposition}[theorem]{Proposition}

\theoremstyle{definition}

\newtheorem{remark}[theorem]{Remark}
\newtheorem{example}[theorem]{Example}

\newcommand{\R}{\mathbb{R}}
\newcommand{\Z}{\mathbb{Z}}
\newcommand{\Sl}{\mathrm{SL}(2,\R)}
\newcommand{\id}{\operatorname{Id}}
\newcommand{\li}{\langle}
\newcommand{\ri}{\rangle}
\newcommand{\ii}{\textit{\textbf{i}}}
\newcommand{\jj}{\textbf{\textit{j}}}
\newcommand{\kk}{\textbf{\textit{k}}}

\newcommand{\Ss}{\mathbb{S}}
\newcommand{\slf}{\mathfrak{sl}(2,\R)}

\newcommand{\SO}{\mathrm{SO}}

\newcommand{\U}{\text U}
\newcommand{\SU}{\text{SU}}
\newcommand{\Psl}{\mathrm{PSL}(2,\R)}

\newcommand{\iso}{\operatorname{Iso}}
\newcommand{\isoo}{\operatorname{Iso}_o}

\makeatletter
\newcommand\columntag[2]{#1\def\@currentlabel{#1}.\label{#2}}
\makeatother

\usepackage{import}
\usepackage{xifthen}
\usepackage{pdfpages}
\usepackage{transparent}

\begin{document}

\begin{abstract}
   We consider the pseudo-nearly Kähler $\Sl\times\Sl$ and we study its Lagrangian submanifolds. We provide examples of Lagrangian submanifolds which do not have an analogue in $\Ss^3\times\Ss^3$. We also provide an expression for the isometry group of $\Sl\times\Sl$ with the pseudo-Riemannian nearly Kähler metric. The main result is a complete classification of extrinsically homogeneous Lagrangian submanifolds in this space.
\end{abstract}
\maketitle
\section{Introduction}

\vspace{3 ex}
Kähler geometry can be seen as the intersection of three areas of differential geometry: symplectic, Riemannian and complex geometry. That is, a Kähler manifold carries a Riemannian metric $g$, a symplectic form $\omega$ and a complex structure $J$ such that 
\[
g(JX,JY)=g(X,Y), \ \ \ \ \ \ \ \ \ \  \omega(X,Y)=g(JX,Y).    
\]
Equivalently, an almost Hermitian manifold $(M,g,J)$ is Kähler if and only if $\nabla J\equiv 0$, where $\nabla$ is the Levi-Civita connection associated to~$g$. 

There are only two spheres admiting an almost complex structure: $\Ss^2$ and $\Ss^6$. 
Surprisingly, the unit six-sphere $\Ss^6$ with its canonical almost complex structure, inherited from the octonion product, is not Kähler. 
Instead, it is what is known as a nearly Kähler manifold.

A nearly Kähler manifold is an almost Hermitian manifold $(M,J,g)$ such that $\nabla J$ is a skew symmetric $(2,1)$-tensor. 
If moreover, $\nabla_XJ\not\equiv0$ for \textit{all} $X\in \mathfrak{X}(M)$, then we say that $M$ is strict nearly Kähler. 
In six dimensions, this is equivalent to being nearly Kähler but not Kähler. 
In general, we can think of strict nearly Kähler manifolds as nearly Kähler manifolds from which we cannot extract Kähler factors. 
Indeed, Gray \cite{gray2} showed that any complete, simply connected nearly Kähler manifold can be written as $M_1\times M_2$ where $M_1$ is strict nearly Kähler and $M_2$ is Kähler. 

Strict nearly Kähler manifolds turned out to be quite rare. 
For instance, in \cite{gray2} Gray also showed that there are no eight-dimensional strict nearly Kähler manifolds. 
Also, it can be easily proved that two- and four-dimensional nearly Kähler manifolds are automatically Kähler. 
Later on, Nagy~\cite{nagy2} showed that any complete 10-dimensional nearly Kähler manifold is either the product of a six-dimensional nearly Kähler manifold and a Kähler surface, or the twistor space over an eight-dimensional quaternionic Kähler manifold with positive Einstein constant (or positive quaternionic Kähler, for short).
All eight-dimensional positive quaternionic Kähler manifolds are classified, being the symmetric spaces $\mathbb{H}P^2$, $\mathbb{G}r_2(\mathbb{C}^4)$ and $G_2/\SO (4)$.
Therefore, all 10-dimensional strict nearly Kähler manifolds are classified, as described in~\cite{lebrunsolomon}.

Six-dimensional strict nearly Kähler manifolds are of particular interest, since they are the lowest-dimensional non-Kähler examples we encounter. 
In fact, Nagy~\cite{nagy} showed that any nearly Kähler manifold is a Riemannian product whose factors are six-dimensional nearly Kähler manifolds, certain homogeneous nearly Kähler manifolds or twistor spaces over positive quaternionic Kähler manifolds.

All six-dimensional Riemannian \textit{homogeneous} strict nearly Kähler manifolds were classified by Butruille in~\cite{butruille}, these being
\begin{equation*}
    \begin{aligned}
        \bullet & \quad  \Ss^6&&=G_2/\SU(3),\\
        \bullet & \quad \Ss^3\times\Ss^3&&=
        (\SU(2)\times\SU(2)\times\SU(2))/\Delta \SU(2),\\
        \bullet & \quad \mathbb{C}P^3&&=\operatorname{Sp}(2)/(\U(1)\times\SU(2)),\\
        \bullet &\quad  F(\mathbb{C}^3)&&=\SU(3)/(\U(1)\times\U(1)).    \\
    \end{aligned}
\end{equation*}

Recently, Foscolo and Haskins~\cite{foscolo} showed the existence of six-dimensional strict nearly Kähler manifolds which are not homogeneous.

By means of a $T$-dual construction, Kath~\cite{ines} and Schäfer~\cite{Schafer} provided six pseudo-Riemannian analogues of the spaces in the list above. 
However, this list does not provide a full classification of six-dimensional homogeneous pseudo-Riemannian nearly Kähler manifolds (or pseudo-nearly Kähler for short). 
In fact, Alekseevsky et al.\ constructed in~\cite{Alekseevsky} an example of a homogeneous pseudo-nearly Kähler six-manifold which is not a $T$-dual of a Riemannian one.

In this article, we focus on the analogue of $\Ss^3\times\Ss^3$, which is the pseudo-nearly Kähler $\Sl\times\Sl$. 
In~\cite{anarella}, the authors studied Lagrangian submanifolds of $\Sl\times\Sl$ and gave a classification up to congruence of all totally geodesic Lagrangian submanifolds. 
There, the authors divided Lagrangian submanifolds into four types, which depend on their behavior with respect to a specific almost product structure. 

In this paper, we provide an expression of the isometry group of $\Sl\times\Sl$ which, to the knowledge of the author, cannot be found in the literature.

\begin{theorem}\label{groupofisometries1}
    The isometry group of the pseudo-nearly Kähler $\Sl\times\Sl$ is $\big(\Sl\times\Sl\times\Sl\big)\rtimes\big(\Z_2\times S_3\big)$, where $S_3$ is the symmetric group of order 6.
\end{theorem}

Moreover, we study extrinsically homogeneous Lagrangian submanifolds of $\Sl\times\Sl$. 
That is, those Lagrangian submanifolds $f:M\to\Sl\times\Sl$ such that there exist a Lie group $H$ acting transitively by isometries on $f(M)$. 
To simplify things we take $H$ to be a Lie subgroup of the connected component of the identity of the isometry group, i.e.\
\[\isoo(\Sl\times\Sl)=\Sl\times\Sl\times\Sl.\]

Orbits of Lie subgroups of the isometry group are among the most natural submanifolds of homogeneous spaces. 
They are tightly linked to submanifolds with constant sectional curvature, totally geodesic submanifolds  and submanifolds with parallel second fundamental form (see for instance \cite{discala,olmos}). 
In particular, Lagrangian submanifolds of six-dimensional (pseudo-)nearly Kähler manifolds are automatically minimal.

In Theorem \ref{maintheorem} we provide a full classification of extrinsically Lagrangian homogeneous submanifolds of $\Sl\times\Sl$.
Among the submanifolds in the classification, we can find the three totally geodesic examples of \cite{anarella} and two submanifolds with constant sectional curvature. 
These five submanifolds have analogues in the classification of extrinsically homogeneous Lagrangian submanifolds of $\Ss^3\times\Ss^3$, found in \cite{constantangles}.

In addition, we obtain three more Lagrangian submanifolds without an analogue in the Riemannian case. 
The first two are immersions of a space form with constant sectional curvature $-\tfrac{3}{2}$.
Moreover, the second example is actually an infinite family of Lagrangian immersions.
Finally, the last example does not seem to have any intrinsic or extrinsic invariant besides of being Lagrangian and minimal.

\begin{theorem}\label{maintheorem}
Let $f:(M,g)\to\Sl\times\Sl$ be an extrinsically homogeneous Lagrangian immersion into the pseudo-nearly Kähler $\Sl\times\Sl$. 
Then $f(M)$ is congruent to an open subset of the image of one of the following embeddings, whose image is the orbit of $(\id_2,\id_2)$ by~$H\subset\isoo(\Sl\times\Sl)$:

\begin{equation*}
    \begin{tblr}{c c c c l}
        \SetHline{1-Z}{0.5 pt}
    (M,g) & f & H & \text{\emph{Isotropy}} & \text{\emph{Remarks}} \\
        \SetHline{1-Z}{0.5 pt}   
     (\Sl, \frac{2}{3}\li,\ri) & u\mapsto (u,u) &\Sl & 0 &  \begin{aligned}
         & \text{\emph{Totally geodesic}} \\  & K=-\tfrac{3}{2}  
     \end{aligned}\\
     (\Sl,g^+_{\kappa,\tau})  & u\mapsto (u,\ii u\ii) & \Sl & 0 & \text{\emph{Totally geodesic}}\\
     (\Sl,g^-_{\kappa,\tau})  & u\mapsto (u,-\kk u\kk) & \Sl & 0 & \text{\emph{Totally geodesic}} \\
     (\Psl, \frac{8}{3}\li,\ri) & [u]\mapsto (\ii u\ii u^{-1},\jj u\jj u^{-1}) &\Sl & \Z_2 & K=-\tfrac{3}{8}\\
     \R^3_1/\mathbb{Z}  & (u,v,w)\mapsto (e^{v \ii}e^{-u \kk},e^{w \jj}e^{-u \kk}) &  \R^2\times \Ss^1 & 0 & K=0\\
     (\R^3,\hat{g}) & \iota  & \R\ltimes_{\varphi_o}\R^2 & 0 & K=-\tfrac{3}{2}\\
     (\R^3/H_\lambda,g_\lambda)  & f_\lambda & (\R\ltimes_{\varphi_1}\R^2)/ H_\lambda & 0 & K=-\tfrac{3}{2}\\
     (\R^3,\tilde{g}) & \jmath & \R\ltimes_{\varphi_2}\R^2 & 0 & \\
    \SetHline{1-Z}{0.5 pt}
\end{tblr}
\end{equation*}
Here $K$ is the sectional curvature of $f(M)$ and $
\ii,\jj,\kk$ are the matrices
\[
\ii=\begin{pmatrix}
1&0\\
0&-1\\
\end{pmatrix}, \ \ \ \ \ \ \jj=\begin{pmatrix}
    0&1\\
    1&0\\
\end{pmatrix}, \ \ \ \ \ \ 
\kk=\begin{pmatrix}
0&1\\
-1&0
\end{pmatrix}.
\]
Furthermore, $\li,\ri$ is the metric given in \eqref{prodsl2}, $g^+_{\kappa,\tau}$ and $g^-_{\kappa,\tau}$ are Berger-like metrics stretched in a spacelike and timelike direction, respectively; $\hat{g}$, $\iota$ and $\varphi_o$ are given in Example \ref{type2asubmanifoldexample}; $g_\lambda$, $f_\lambda$, $H_\lambda$ and $\varphi_1$ are given in Example \ref{type2bsubexample}; $\tilde{g}$, $\jmath$ and $\varphi_2$ are given in Example \ref{type3submanifold}.

Conversely, the maps listed in the table above are extrinsically homogeneous Lagrangian submanifolds of $\Sl\times\Sl$. 
Moreover, all immersions are not congruent to each other, including the different immersions of the family $f_\lambda$.
\end{theorem}
As the definition indicates, classifications of extrinsically homogeneous submanifolds usually follow from  classifications of Lie subgroups of the isometry group.
However, as opposed to the compact case, an isometry group with non-compact semi-simple Lie algebra might present many difficulties when looking for subgroups, as maximal Lie subalgebras might not be reductive (see~\cite{mostow}).
The condition of being reductive simplifies the process considerably, as we can obtain its maximal Lie subalgebras in a simple way (see Theorem 2.1 in~\cite{kollross}).
Moreover, when we increase the codimension of the submanifold, the process gets more difficult, as we may have to repeat it several times.

Therefore, in this paper we classify extrinsically homogeneous Lagrangian submanifolds by using the properties of Lagrangian submanifolds of nearly Kähler spaces, and how the Lie subgroup acting on them preserves their structure.

The paper is organized as follows.
In Section \ref{preliminaries} we give a brief introduction to the pseudo-nearly Kähler structure of $\Sl\times\Sl$ and we prove Theorem \ref{groupofisometries1}. 
In Section \ref{lagrangiansubmanifolds} we state some properties of Lagrangian submanifolds of $\Sl\times\Sl$. 
In Section \ref{sectionextrhomosub} we provide examples of extrinsically homogeneous Lagrangian submanifolds of $\Sl\times\Sl$.
Finally, in Section \ref{sectionproofmaintheorem} we prove Theorem \ref{maintheorem}.

\section{The pseudo-nearly Kähler \texorpdfstring{$\Sl\times\Sl$}{SL(2,R)xSl(2,R)}} \label{preliminaries}
The nearly Kähler structure of $\Sl\times\Sl$ is given in detail in \cite{anarella} and \cite{Ghandour}. Here we recall some of the structure necessary for this article.

\subsection{The manifold \texorpdfstring{$\boldsymbol{\Sl}$}{SL(2,R)}}
Consider the  real vector space of $2\times2$ real matrices $M(2,\R)$ with the indefinite inner product $\li,\ri$ given by
\begin{equation}
    \langle a,b\rangle=-\frac{1}{2}\operatorname{Trace}(\operatorname{adj}(a)b).\label{prodsl2}
\end{equation}
The real special linear group $\Sl$ can be defined as
\begin{equation}\label{slspaceform}
    \Sl=\{a\in M(2,\R):\langle a,a\rangle=-1\}.     
\end{equation}
Identifying $(M(2,\R),\li,\ri)$ with $\R_2^4$ we obtain that $\Sl$ is isometric to the three-dimensional anti-de Sitter space $H^3_1(-1)$, defined as
\[
H_1^3(-1)=\left\{x\in\R^4_2:-x_0^2-x_1^2+x_2^2+x_3^2=-1\right\}.    
\]
Hence, $(\Sl,\li,\ri)$ is a Lorentzian manifold with constant sectional curvature $-1$. 
From \eqref{slspaceform} it follows that the tangent space of $\Sl$ at a matrix $a$ is the orthogonal space $a^\perp$, which can be written as
\begin{equation*}
T_a\Sl =a^\perp =\{ a\alpha: \alpha\in \mathfrak{sl}(2,\R)\}.   \label{tangentspace}
\end{equation*}
Here $\mathfrak{sl}(2,\R)$ is the Lie algebra of $\Sl$, consisting of all the matrices in $M(2,\R)$ with vanishing trace.

Similar to the tangent space of the three-sphere $\Ss^3$, the tangent space of $\Sl$ at $\id_2$ is spanned by the split-quaternions $\ii$, $\jj$ and $\kk$, given by
\begin{equation}
\ii=\begin{pmatrix}
    1&0\\
    0&-1
\end{pmatrix}, \ \ \ \ \ \ \jj=\begin{pmatrix}
0&1\\
1&0\\
\end{pmatrix}, \ \ \ \ \ \
\kk=\begin{pmatrix}
0&1\\
-1&0\\
\end{pmatrix}.\label{ijksl2R}
\end{equation}
We define the frame $\{X_i\}_i$ on $\Sl$ by
\begin{equation}
    X_1(a)=a\ii, \ \ \ \ \ \ \ \ \ X_2(a)=a\jj, \ \ \ \ \ \ \ \ \ X_3(a)=a\kk. \label{frameXi}
\end{equation}
This is an orthogonal frame on $\Sl$ with
\[
\li X_1,X_1\ri=\li X_2 ,  X_2\ri=1, \ \ \ \ \ \ \ \li X_3,X_3\ri=-1.  
\]
Given $\alpha$ and $\beta$ in $\slf$, we have 
\begin{equation}
\alpha\beta=\alpha\times\beta+\li\alpha,\beta\ri\id_2,    \label{prodinslf}
\end{equation}
where $\alpha\times\beta=\tfrac{1}{2}(\alpha\beta-\beta\alpha)$.

\subsection{The homogeneous nearly Kähler structure on \texorpdfstring{$\boldsymbol{\Sl\times\Sl}$}{SL(2,R)xSL(2,R)}}
Consider the triple product $\Sl\times\Sl\times\Sl$ with the product metric arising from $\li,\ri$ given in \eqref{prodsl2}. Let $\pi\colon\Sl \times\Sl\times\Sl\to\Sl\times\Sl$ be the submersion given by
\[
\pi (a,b,c)=(ac^{-1},bc^{-1}).    
\]
Let $g$ be the metric on $\Sl\times\Sl$ such that $\pi$ is a pseudo-Riemannian submersion. Then, $(\Sl\times\Sl,g)$ is a pseudo-Riemannian homogeneous manifold, expressed as
\[
\Sl\times\Sl=\frac{\Sl\times\Sl\times\Sl}{\Delta\Sl},    
\]
where $\Delta \Sl=\{(a,a,a):a\in\Sl\}$.

We define an almost complex structure $J$ on $\Sl\times\Sl$ by

\begin{equation}
    J(a\alpha,b\beta)=\frac{1}{\sqrt{3}} (a(\alpha-2\beta),b(2\alpha-\beta)),    \label{defJ}
\end{equation}
for $\alpha,\beta\in\mathfrak{sl}(2,\R)$. This almost complex structure satisfies
\begin{equation}
    J^2=-\id, \ \ \ \ \ \ \ \ \ \ \ \ \ g(X,Y)=\li X,Y\ri+\li JX,JY\ri,\label{gcompatiblewithJ}
\end{equation}
where $\li,\ri$ the product metric of $\Sl\times\Sl$ associated to the metric given in (\ref{prodsl2}).
The last equality implies that $J$ is compatible with $g$. 
Hence we may say that $(\Sl\times\Sl,g,J)$ is an almost Hermitian manifold.
Using \eqref{gcompatiblewithJ}, we get the following explicit expression for $g$:
\begin{equation*}
    g((a\alpha,b\beta),(a\gamma,b\delta))=\frac{2}{3}\langle(a\alpha,b\beta),(a\gamma,b\delta)\rangle-\frac{1}{3}\langle(a\beta,b\alpha),(a\gamma,b\delta)\rangle,
\end{equation*}
with $\alpha,\beta,\gamma,\delta\in\mathfrak{sl}(2,\R) $.

We denote by $\tilde{\nabla}$ the Levi-Civita connection on $\Sl\times\Sl$ associated with the pseudo-nearly Kähler metric $g$. 
We denote the covariant derivative of $J$ by $G$, which turns out to be skew symmetric.  
Hence, $(\Sl\times\Sl,g,J)$ is a pseudo-nearly Kähler manifold.
As such, $G$ satisfies
\begin{equation}
    g(G(X,Y),Z)+g(G(X,Z),Y)=0, \ \ \ \ \ G(X,JY)+JG(X,Y)=0,\label{nkprop}
\end{equation}
where
$X$, $Y$ and $Z$ are vector fields on $\Sl\times\Sl$.
These equations imply that
\begin{equation}
    g(G(X,Y),JZ)+g(G(X,Z),JY)=0.
    \label{gnormal}
\end{equation}

Consider now the almost product structure $P$ on $\Sl\times\Sl$ given by
\begin{equation}
    P(a\alpha,b\beta)=(a\beta,b\alpha). \label{prodstructuredef}
\end{equation}
The tensor $P$ has the following properties:
\begin{equation}
    \begin{aligned}
        &P^2=\id, && g(PX,PY)=g(X,Y),\\
        &PJ=-JP, && g(PX,Y)=g(X,PY),\\
        &PG(X,Y)+G(PX,PY)=0. &&
    \end{aligned}\label{proppe}
\end{equation}
for any pair of vector fields $X,Y$ on $\Sl\times\Sl$.
The covariant derivative of $P$ can be expressed in terms of $P$, $J$ and $G$ in the following way:
\begin{equation}
(\Tilde{\nabla}_XP)Y=\frac{1}{2}(JG(X,PY)+JPG(X,Y)). \label{nablap}
\end{equation}

Six-dimensional nearly Kähler manifolds carry a distinguished constant, known as the \textit{type}. In particular, $(\Sl\times\Sl,g,J)$ has type $-\tfrac{2}{3}$. Namely, $G$ satisfies the formula
\begin{equation}
    \begin{split}
     g(G(X,Y),G(Z,W))=-\tfrac{2}{3}\big(&g(X,Z)g(Y,W)-g(X,W)g(Y,Z)\\
     &+g(JX,Z)g(Y,JW)-g(JX,W)g(Y,JZ)\big).
    \end{split}
    \label{constanttype}
    \end{equation}

The curvature tensor $\tilde{R}$ of $\Sl\times\Sl$ associated to the Levi-Civita connection $\tilde{\nabla}$ of the pseudo-nearly Kähler metric $g$  is given by
\begin{equation}
        \begin{split}
            \Tilde{R}(U,V)W=-\tfrac{5}{6}\Big(&g(V,W)U-g(U,W)V\Big)\\
            -\tfrac{1}{6}\Big(&g(JV,W)JU-g(JU,W)JV-2g(JU,V)JW\Big)\\
            -\tfrac{2}{3}\Big(&g(PV,W)PU-g(PU,W)PV\\
            &+g(JPV,W)JPU-g(JPU,W)JPV\Big).\label{curv}
        \end{split}
\end{equation}

\subsection{The manifold \texorpdfstring{$\boldsymbol{\Sl\times\Sl}$}{SL(2,R)xSL(2,R)} as a pseudo-Riemannian product}
Take the product metric $\li,\ri$ on $\Sl\times\Sl$ arising from the Lorentzian metric on $\Sl$ given in (\ref{prodsl2}). The nearly Kähler metric $g$ and $\li,\ri$ are related by
\begin{equation}
    \li X,Y\ri=2 g(X,Y)+g(X,PY), \label{prodmetric}
\end{equation}
Products of (pseudo-)Riemannian manifolds carry a canonical almost product structure $Q$ compatible with the product metric, given by
\[
Q(X_1,X_2)=(-X_1,X_2).
\]
Similar to the relation between the product metric and the nearly Kähler metric, $P$ and $Q$ are related by
\begin{equation}
    QX=-\frac{1}{\sqrt{3}}(2PJX-JX).\label{prodQ}
\end{equation}

Consider the immersion of $\Sl\times\Sl$ with the product metric $\li,\ri$ into $M(2,\R)\times M(2,\R)\cong\R^8_4$.
Let $D$ be the Euclidean connection of $\R^8_4$. The Gauss formula splits $D$ into tangent and normal parts:
\[
D_XY=\nabla^E_XY+h^E(X,Y),
\]
with $X,Y$ vector fields on $\Sl\times\Sl$. The connection $\nabla^E$ is the Levi-Civita connection associated to the product metric $\li,\ri$ and $h^E$ is the so-called second fundamental form.
We have an expression for $\nabla^E$ in terms of the connection $\tilde\nabla$ associated to $g$, $J,P$ and $G$:
\begin{equation}
    \nabla^E_XY=\tilde{\nabla}_XY+\frac{1}{2}(JG(X,PY)+JG(Y,PX)).\label{relprodkal}
\end{equation}
For $(a,b)\in\Sl\times\Sl$ we have
\begin{equation*}
\begin{split}
        h^E(X,Y)_{(a,b)}=\frac{1}{2}\langle X,Y\rangle(a,b)+\frac{1}{2}\langle Y,QX\rangle(-a,b).
\end{split}
\end{equation*}
Hence,
\begin{equation}
\begin{split}
        D_XY&=\nabla_X^EY+\frac{1}{2}\langle X,Y\rangle(a,b)+\frac{1}{2}\langle Y,QX\rangle(-a,b). \label{connectionr8}
\end{split}
\end{equation}

\subsection{The isometry group}\label{seciso}

The connected component of the identity of the isometry group of $\Sl\times\Sl$ is 
\[\isoo(\Sl\times\Sl)=\Sl\times\Sl\times\Sl,\] 
where an element $\phi_{(a,b,c)}$ acts on a point $(p,q)$ by $\phi_{(a,b,c)}(p,q)=(apc^{-1},bqc^{-1})$.

The isometries $\phi_{(a,b,c)}$ preserve $P$ and $J$, in the sense that $d\phi_{(a,b,c)}\circ J=J\circ d\phi_{(a,b,c)}$ and $d\phi_{(a,b,c)}\circ P=P\circ d\phi_{(a,b,c)}$. 
These isometries are not the only ones that satisfy these properties. 
Given three matrices $a$, $b$ and $c$ with determinant $-1$, the map $(p,q)\mapsto (apc^{-1},bqc^{-1})$ is also an isometry that preserves $J$ and $P$.  

Denote by $\mathrm{SL}^{\pm}(2,\R)$ the group of all matrices in $M(2,\R)$ with determinant $\pm1$.
We can write any matrix of $\mathrm{SL}^{\pm}(2,\R)$ as $\ii^k a$, where $\ii$ is the matrix given in~\eqref{ijksl2R}, $k\in\{0,1\}$ and $a\in\Sl$. 
Thus, we have  $\left(\Sl\times\Sl\times\Sl\right)\rtimes\Z_2\subset\iso(\Sl\times\Sl)$. 

Permutations of elements of $\Sl\times\Sl\times\Sl$ also give rise to isometries of the pseudo-nearly Kähler $\Sl\times\Sl$:
\begin{equation}
    \label{isoslsl}
    \begin{alignedat}{2}
        &\Psi_{0,0}(p,q)=(p,q), 
        &&\Psi_{1,0}(p,q)=(q,p),\\
        &\Psi_{0,2\pi/3}(p,q)=(p q^{-1},q^{-1}),
        &&\Psi_{1,2\pi/3}(p,q)=(q^{-1},p q^{-1}),\\
        &\Psi_{0,4\pi/3}(p,q)=(q p^{-1},p^{-1}),\qquad\qquad
        &&\Psi_{1,4\pi/3}(p,q)=(p^{-1},q p^{-1}).
    \end{alignedat}
\end{equation}
These isometries are not included in $\Sl\times\Sl\times\Sl$. Moreover, each one of these is in a different connected component of $\iso(\Sl\times\Sl)$ and satisfies
\[
    J \circ d\Psi_{\kappa,\tau}=(-1)^\kappa  d\Psi_{\kappa,\tau}\circ J,
\ \ \ \ \ \ \ P\circ d \Psi_{\kappa,\tau}=d\Psi_{\kappa,\tau}\circ(\cos\tau P+\sin \tau J P).
\]
Later on, we prove that these are \textit{all} the isometries of the nearly Kähler $\Sl\times\Sl$.

A key result in the classification of Riemannian homogeneous nearly Kähler manifolds by Butruille \cite{butruille} is the existence of a unique nearly Kähler structure on $\Ss^3\times\Ss^3$.
Consequently, the almost complex structure on $\Ss^3\times\Ss^3$ is unique up to sign.

However, in \cite{Schafer} it is shown that $\Sl\times\Sl$ has a unique \textit{left invariant} nearly Kähler structure, which does not necessarily imply that $J$ is unique up to sign.
Therefore, by isometry group of the pseudo-nearly Kähler $\Sl\times\Sl$ we mean the set of all diffeomorphisms preserving the almost Hermitian strucuture. 
That is, those isometries $\mathcal{F}$ of $(\Sl\times\Sl,g)$ that preserve $J$, i.e.\ $
\mathcal{F}_* J=\pm J\mathcal{F}_*$.

\begin{lemma}\label{threeps}
     Any almost product structure $\tilde{P}$ on $\Sl\times\Sl$ that satisfies all properties in \eqref{proppe} and \eqref{curv} is given by
    \begin{equation}
    \tilde{P}=\cos \eta P+\sin \eta JP, \label{ptilde}
    \end{equation}
    where $\eta$ is equal to $\tfrac{2\pi}{3}$ or $\tfrac{4\pi}{3}$ and $P$ is given in \eqref{prodstructuredef}.
    Conversely, if $\tilde{P}$ is given by \eqref{ptilde}, then it satisfies \eqref{proppe} and \eqref{curv}. Moreover, it also satisfies \eqref{nablap}.
\end{lemma}
\begin{proof}
    A similar result was proven for the Riemannian analogue $\Ss^3\times\Ss^3$ in \cite{moruz_properties_2018}. We can follow the same proof for $\Sl\times\Sl$.
\end{proof}
The following lemma is a well known result.
\begin{lemma}\label{lemmasl2requaltoso21}
    Let $\{\alpha_1,\alpha_2,\alpha_3\}$ and $\{\beta_1,\beta_2,\beta_3\}$ be bases of $\slf$. If $\li\alpha_i,\alpha_j\ri=\li \beta_i,\beta_j\ri$ for all $i$, $j\in \{1,2,3\}$, then there exists a matrix $c$ in $\mathrm{SL}^{\pm}(2,\R)$ such that $c\alpha_i c^{-1}=\beta_i$. In other words, $\mathrm{SL}^{\pm}(2,\R)/\Z_2$ is isomorphic to $\SO(2,1)$.
\end{lemma}
With these lemmas we prove the following statement.
\begingroup
\def\thetheorem{\ref{groupofisometries1}}

\begin{theorem}
    The isometry group of the pseudo-nearly Kähler $\Sl\times\Sl$ is the semi-direct product $\big(\Sl\times\Sl\times\Sl\big)\rtimes\big(\Z_2\times S_3\big)$, where $S_3$ is the symmetric group of order 6 generated by $\{\Psi_{1,0},\Psi_{1,4\pi/3}\}$
\end{theorem}
\addtocounter{theorem}{-1}
\endgroup

\begin{remark}
    An element $(a,b,c,\Psi,k)$ acts on a point $(p,q)$ by
    \[
        (a,b,c,\Psi,k)\cdot(p,q)=\Psi\circ\phi_{\ii^k( a, b, c)}(p,q).
    \]
\end{remark}
\begin{proof}[Proof of Theorem \ref{groupofisometries1}]
    We already know that the given group is included in $\iso(\Sl\times\Sl)$. 
    Here we show the oposite inclusion.

    Let $\mathcal{F}$ be an isometry of the pseudo-nearly Kähler $\Sl\times\Sl$. 
    Thus, there exist $\kappa_0\in\{0,1\}$ satisfying
    \[
        \mathcal{F}_*J=(-1)^{\kappa_0}J\mathcal{F}_*.
    \]
    As $\mathcal{F}_*P(\mathcal{F}^{-1})_*$ is an almost product structure satisfying \eqref{proppe} and \eqref{curv},
    Lemma~\ref{threeps} implies that
    \[
        \mathcal{F}_*P(\mathcal{F}^{-1})_*=\cos \tau_0 P +\sin \tau_0 JP,
    \]
    for some $\tau_0\in\{0,\tfrac{2\pi}{3},\tfrac{4\pi}{3}\}$.
    By taking the composition $\mathcal{F}\circ \Psi_{\kappa_0,(-1)^{\kappa_0}\tau_0}$ we may assume that $\mathcal{F}$ preserves $P$ and $J$. 
    Let $(p_o,q_o)\in\Sl\times\Sl$ such that $\mathcal{F}(\id_2,\id_2)=(p_o,q_o)$. 
    Then by taking the composition $\mathcal{F}\circ\phi_{(p_o^{-1},q_{o}^{-1},\id_2)}$ we may also assume that $\mathcal{F}(\id_2,\id_2)=(\id_2,\id_2)$. 

    Let $\alpha\in\slf$. Then we write ${\mathcal{F}_*}_{(\id_2,\id_2)}(\alpha,0)=(\beta,\gamma)$. Since $\mathcal{F}$ preserves $P$ we know that ${\mathcal{F}_*}_{(\id_2,\id_2)}(0,\alpha)=(\gamma,\beta)$. We compute
    \begin{equation*}
        \begin{split}
            {\mathcal{F}_*}_{(\id_2,\id_2)}J(\alpha,0)&=\frac{1}{\sqrt{3}}{\mathcal{F}_*}_{(\id_2,\id_2)}(\alpha,2\alpha)\\
            &=\frac{1}{\sqrt{3}}{\mathcal{F}_*}_{(\id_2,\id_2)}(\alpha,0)+\frac{1}{\sqrt{3}}{\mathcal{F}_*}_{(\id_2,\id_2)}(0,2\alpha)\\
            &=\frac{1}{\sqrt{3}}(\beta,\gamma)+\frac{2}{\sqrt{3}}(\gamma,\beta)\\
            &=\frac{1}{\sqrt{3}}(\beta+2\gamma,2\beta+\gamma).
        \end{split}
    \end{equation*}
    On the other hand, as $\mathcal{F}$ preserves $J$, we have that ${\mathcal{F}_*}_{(\id_2,\id_2)}J(\alpha,0)$ equals 
    \begin{equation*}
        \begin{split}
            J{\mathcal{F}_*}_{(\id_2,\id_2)}(\alpha,0)&=J(\beta,\gamma)\\
            &=\frac{1}{\sqrt{3}}(\beta-2\gamma,2\beta-\gamma).
        \end{split}
    \end{equation*}
    Therefore, we obtain that $\gamma=0$. 
    Moreover, since $\mathcal{F}$ is an isometry, we deduce that ${\mathcal{F}_*}$ maps a set $\{(\alpha_1,0),(\alpha_2,0),(\alpha_3,0)\}$ to a set $\{(\beta_1,0),(\beta_2,0),(\beta_3,0)\}$ such that $\li \alpha_i,\alpha_j\ri=\li \beta_i,\beta_j\ri$.

    Now, using Lemma \ref{lemmasl2requaltoso21}, we may compose $\mathcal{F}$ with an isometry of $\Sl\times\Sl\times\Sl\rtimes\Z_2$ to assume that ${\mathcal{F}_*}_{(\id_2,\id_2)}(\alpha,0)=(\alpha,0)$ for all $\alpha\in\slf$. Since $\mathcal{F}$ preserves $P$, we have that
    \begin{equation*}
        \begin{split}
            {\mathcal{F}_*}_{(\id_2,\id_2)}(\alpha,\beta)&={\mathcal{F}_*}_{(\id_2,\id_2)}(\alpha,0)+{\mathcal{F}_*}_{(\id_2,\id_2)}(0,\beta)\\
                                            &=(\alpha,0)+P(\beta,0)\\
                                            &=(\alpha,\beta).\\
        \end{split}
    \end{equation*}

    Since isometries are determined by a point and the differential at that point, the argument above shows that $\mathcal{F}^{-1}$ is in $\big(\Sl\times\Sl\times\Sl\big)\rtimes\big(\Z_2\times S_3\big)$, hence $\mathcal{F}$ also belongs to this group.
\end{proof}

\section{Lagrangian submanifolds of \texorpdfstring{$\Sl\times\Sl$}{SL(2,R)xSL(2,R)}}\label{lagrangiansubmanifolds}
Let $f:M\to\Sl\times\Sl$ be a non-degenerate pseudo-Riemannian immersion. The Gauss formula relates the Levi-Civita connection on $M$ and the nearly Kähler connection $\tilde{\nabla}$ by
\begin{equation}
    \tilde{\nabla}_XY=\nabla_XY+h(X,Y),\label{gaussformula}
\end{equation}
where $X,Y$ are vector fields on $M$ and $h$ is a symmetric bilinear normal form called the second fundamental form. 
On the other hand, the Weingarten formula gives a relation between $\tilde{\nabla}$ and the normal connection:
\[
 \tilde{\nabla}_X\xi=-S_{\xi}X+\nabla_X^{\bot}\xi,
\]
where $X$ and $\xi$ are tangent and normal vector fields on $M$, respectively.
Given $\xi$ a normal vector field to $M$, the tensor $S_\xi$ is called the shape operator.
It is linear at $\xi$ and $X$ and symmetric (with respect to $g$), related to $h$ by $g(S_\xi X,Y)=g(h(X,Y),\xi)$.

A totally geodesic submanifold is a submanifold whose geodesics are also geodesics of $\Sl\times\Sl$. This is equivalent to a vanishing second fundamental form.

The most fundamental equations of submanifold theory are the Gauss and Codazzi equations, which give expressions for the tangent and normal parts of the curvature tensor, respectively. Namely, 
\begin{equation*}
    \begin{split}
        \left(\tilde{R}(X,Y)Z\right)^{\top}&=R(X,Y)Z+S_{h(X,Z)}Y-S_{h(Y,Z)}X,\\
        \left(\tilde{R}(X,Y)Z\right)^{\bot}&=(\overline{\nabla}_Xh)(Y,Z)-(\overline{\nabla}_Yh)(X,Z),
    \end{split}
\end{equation*}
where $(\overline{\nabla}_Xh)(Y,Z)=\nabla_X^{\bot}h(Y,Z)-h(\nabla_XY,Z)-h(Y,\nabla_XZ)$ and $\tilde{R},R$ are the curvature tensors of $\Sl\times\Sl$ and $M$, respectively.

Assume now that $M$ is a Lagrangian submanifold. That is, a non-degenerate three-dimensional submanifold such that $J$ maps tangent spaces of $M$ into normal spaces, and vice versa.

In \cite{Schafer} we find properties of the tensor $G$ of a pseudo-nearly Kähler manifold and the second fundamental form of a Lagrangian submanifold:
\begin{equation}
    \begin{split}
    g(G(X,Y),Z)&=0,\\
    g(h(X,Y),JZ)&=g(h(X,Z),JY).\\
    S_{JX}Y&=Jh(X,Y)
    \end{split}
    \label{lagrprop}
\end{equation}
Lagrangian submanifolds are particularly interesting, since their behavior with respect to the almost Hermitian structure, and because of the following property, also found in \cite{Schafer}:
\begin{proposition}
Any Lagrangian submanifold of a six-dimensional strictly pseudo-nearly Kähler  \hyphenation{ma-ni-fold} manifold is orientable and minimal. \label{minimal}
\end{proposition}

The tangent bundle of $\Sl\times\Sl$ splits into the tangent and normal bundles of a submanifold $M$. Moreover, if $M$ is Lagrangian, we have that $T\Sl\times\Sl=TM\oplus JTM$. 
Hence, there exist two endomorphisms $A,B:TM\rightarrow TM$ such that $P|_M=A+JB$, where $P$ is the almost product structure given in \eqref{prodstructuredef}. From Equation (\ref{proppe}) we deduce that $A$ and $B$ are symmetric with respect to $g$, commute with each other and $A^2+B^2=\id$. 
Then, the Gauss and Codazzi equations follow from (\ref{curv}) as:
\begin{equation}
    \begin{split}
        R( X,Y)Z&=-\tfrac{5}{6}\big(g(Y,Z) X-g( X,Z)Y\big)\\
            &\quad -\tfrac{2}{3}\big(g(AY,Z)A X-g(A X,Z)AY+g(BY,
            Z)B X\\
            &\quad-g(B X,Z)BY\big)-S_{h( X,Z)}Y+S_{h(Y,Z)} X,
            \label{Gauss}
    \end{split}
\end{equation}
\begin{equation}
\begin{split}
    (\overline{\nabla}_ X h)(Y,Z)-(\overline{\nabla}_Y h)( X,Z)&=-\tfrac{2}{3}\big(g(AY,Z)JB X-g(A X,Z)JBY\\
    &\quad-g(BY,Z)JA X+g(B X,Z)JAY\big).\label{Codazzi}
\end{split}
\end{equation}

A $\Delta_i$-orthonormal basis of $\R^3_1$ is a basis $\{e_1,e_2,e_3\}$ such that the matrix of inner products is given by $\Delta_i$, where
\begin{equation*}
\Delta_1=\begin{pmatrix}
    -1 & 0 & 0 \\
    0 & 1 & 0 \\
    0  & 0 & 1 \\
\end{pmatrix},\ \ \ \
\Delta_2=\begin{pmatrix}
    0 & 1 & 0 \\
    1 & 0 & 0 \\
    0  & 0 & 1 \\
\end{pmatrix},\ \ \ \ 
\Delta_3=\begin{pmatrix}
    1 & 0 & 0 \\
    0 & -1 & 0 \\
    0  & 0 & 1 \\
\end{pmatrix} .
\end{equation*}
Given a $\Delta_i$-orthonormal basis of $T_pM$ we can extend this concept to what we call a $\Delta_i$-orthonormal frame.

Given a $\Delta_i$-orthonormal frame $\{E_1,E_2,E_3\}$ on a Lagrangian submanifold $M$ of $\Sl\times\Sl$, by (\ref{lagrprop}) we have that $G(E_i,E_j)$ is a normal vector field. Moreover, by (\ref{gnormal}) and (\ref{constanttype}) we have that $JG(E_j,E_k)=\alpha E_l$ where $\alpha$ depends on $\Delta_i$, as indicated in the following table.
\begin{equation}
    \begin{tblr}{|c|c|c|c|}
\hline
      &   \Delta_1 &   \Delta_2  & \Delta_3  \\   
     \hline
     JG(E_1,E_2)  &  \sqrt{\frac{2}{3}}E_3  &  \sqrt{\frac{2}{3}}E_3 & \sqrt{\frac{2}{3}}E_3 \\
     JG(E_1,E_3)  &  -\sqrt{\frac{2}{3}}E_2  &  -\sqrt{\frac{2}{3}}E_1  & \sqrt{\frac{2}{3}}E_2 \\
     JG(E_2,E_3)  &  -\sqrt{\frac{2}{3}}E_1  &  \sqrt{\frac{2}{3}}E_2   & \sqrt{\frac{2}{3}}E_1 \\
     \hline
\end{tblr}\label{tabla}
\end{equation}

\begin{lemma}[\cite{anarella}]
    Let $M$ be a Lagrangian submanifold of the pseudo-nearly Kähler  $\Sl\times\Sl$ and $P$ the almost product structure given in $(\ref{prodstructuredef})$. 
    Then there exists a Lagrangian submanifold $N$ congruent to $M$ such that the restriction of $P$ to $N$ can be written as $P|_N=A+JB$, where $A,B:TN\to TN$ must have one of the following forms, with respect to a $\Delta_i$-orthonormal frame $\{E_1,E_2,E_3\}$:
   \begin{table}[H]
     \centering
     \makebox[\textwidth]{
     \begin{tabular}{l l l}
       \columntag{Type I}{case:10.1} &
       $A = \begin{pmatrix}
       \cos 2\theta_1 & 0 & 0 \\
       0 & \cos 2\theta_2 & 0 \\
       0 & 0 & \cos 2\theta_3
       \end{pmatrix}$, &
       $B = \begin{pmatrix}
       \sin 2\theta_1 & 0 & 0 \\
       0 & \sin 2\theta_2 & 0 \\
       0 & 0 & \sin 2\theta_3
       \end{pmatrix}$, 
       \\[4ex]
       \multicolumn{3}{l}{
       with $\Delta_i=\Delta_1$ and $\theta_1 + \theta_2 + \theta_3 = 0$ modulo $\pi$,}
       \\[2ex]
        \columntag{Type II}{case:10.2} &
       $A = \begin{pmatrix}
       \cos 2\theta_1 & 1 & 0 \\
       0 & \cos 2\theta_1 & 0 \\
       0 & 0 & \cos 2\theta_2
       \end{pmatrix}$, &
       $B = \begin{pmatrix}
       \sin 2\theta_1 & -\cot 2\theta_1 & 0 \\
       0 & \sin 2\theta_1 & 0 \\
       0 & 0 & \sin 2\theta_2
       \end{pmatrix}$, 
       \\[4ex]
       \multicolumn{3}{l}{
        with $\Delta_i=\Delta_2$, $2\theta_1 + \theta_2 = 0$ modulo $\pi$ and $\theta_1 \ne 0, \pi/2$,}
       \\[2ex]
       \columntag{Type III}{case:10.3} &
       $A = \begin{pmatrix}
       -\frac12 & 0 & 1 \\ 0 & -\frac12 & 0 \\ 0 & 1 & -\frac12
       \end{pmatrix}$, &
       $B = \pm \begin{pmatrix}
       \frac{\sqrt 3}{2} & \frac{-4}{3\sqrt 3} & \frac{1}{\sqrt 3} \\
       0 & \frac{\sqrt 3}{2} & 0 \\
       0 & \frac{1}{\sqrt 3} & \frac{\sqrt 3}{2}
       \end{pmatrix}$,
       \\[4 ex]
       \multicolumn{3}{l}{with $\Delta_i=\Delta_2$,}
       \\[2ex]
       \columntag{Type IV}{case:10.4} &
       $A = \begin{pmatrix}
       \cosh\psi \cos 2\theta_1 & \sinh\psi\sin\theta_2 & 0 \\
       -\sinh\psi\sin\theta_2 & \cosh\psi \cos 2\theta_1 & 0 \\
       0 & 0 & \cos 2\theta_2
       \end{pmatrix}$, & \\[4ex]
       & $B = \begin{pmatrix}
       \cosh\psi \sin 2\theta_1 & \sinh\psi\cos\theta_2 & 0 \\
       -\sinh\psi\cos\theta_2 & \cosh\psi \sin 2\theta_1 & 0 \\
       0 & 0 & \sin 2\theta_2
       \end{pmatrix}$,&
       \\[4ex]
       \multicolumn{3}{l}{ with $\Delta_i=\Delta_3$, $2\theta_1 + \theta_2 = 0$ modulo $\pi$, $\theta_2 \ne 0,\pi$ and $\psi \ne 0$.}
     \end{tabular}
     }
   \end{table}
   \label{propAB}
    \end{lemma}
The functions $\theta_i$ and $\psi$ are called the angle functions. Given the extrinsically invariant nature of the type, we say that a Lagrangian submanifold $M$ is of type I, II, III or IV. Likewise, if $M$ is of type i we say that $A$ and $B$ take type i form on $M$.

Suppose $M$ and $N$ are congruent Lagrangian submanifolds of $\Sl\times\Sl$. Namely, there exists $\mathcal{F}\in\iso(\Sl\times\Sl)$ such that $\mathcal{F}(M)=N$. The following lemma provides us a way to see how $P$ projects to $A$ and $B$ on $N$ in terms of $A$ and $B$ on $M$.
\begin{lemma} \label{isometrieswithAB}
    Let $M$ be a Lagrangian submanifold of $\Sl\times\Sl$ and $P$ the almost product structure given in $(\ref{prodstructuredef})$. Assume that $A$ and $B$ are such that $P\vert_M=A+JB$. Let $\mathcal{F}$ be an isometry of the pseudo-nearly Kähler $\Sl\times\Sl$. Then $P\vert_{\mathcal{F}(M)}=\tilde{A}+J\tilde{B}$ with
    \begin{equation*}
        \begin{split}
             \tilde{A}&=\mathcal{F}_*(\cos \tau A+(-1)^\kappa \sin \tau B)\mathcal{F}^{-1}_*\\
             \tilde{B}&=\mathcal{F}_*(-\sin\tau A+(-1)^\kappa \cos \tau B)\mathcal{F}^{-1}_*\\
        \end{split}
    \end{equation*}
    where $\tau$ and $\kappa$ are such that $\mathcal{F}_*  P=(\cos \tau P+\sin\tau JP)\mathcal{F}_*$ and $\mathcal{F}_ *  J=(-1)^\kappa  \mathcal{F}_*$.
\end{lemma}

Let $M$ be a Lagrangian submanifold of $\Sl\times\Sl$ and let $\{E_1,E_2,E_3\}$ be a $\Delta_i$-orthonormal frame such that $A$ and $B$ take type I, II, III or IV form.
We define the functions $\omega_{ij}^k$ and $h_{ij}^k$ by
\[
    \nabla_{E_i}E_j=\sum_{k=1}^3\omega_{ij}^kE_k, \hspace{2 cm} h(E_i,E_j)=\sum_{k=1}^3h_{ij}^kJE_k,
\]
where $\nabla$ is the Levi-Civita connection of the submanifold and $h$ the second fundamental form.
The symmetry of $h$, the second equation of (\ref{lagrprop}) and the compatibility of the connection with the metric yield many symmetries on $\omega_{ij}^k$ and $h_{ij}^k$, depending on $\Delta_l$.

First, for $\Delta_1$- and $\Delta_3$-orthonormal frames $\{E_i\}$  we have
\[
\delta_k\omega_{ij}^k=-\delta_j\omega_{ik}^j, \ \ \ \ \ \ \ h_{ij}^k=h_{ji}^k=\delta_j\delta_kh_{ik}^j,
\]
where 
\[\delta_i=g( E_i,E_i).\]
This implies that $\omega_{ij}^j=0$ for all $i,j=1,2,3$.

For $\Delta_2$-orthonormal frames we obtain 
\[
\omega_{ij}^k=-\omega_{i\widehat{k}}^{\widehat{j}},\ \ \ \ \ \ \ h_{ij}^k=h_{ji}^k=h_{i\widehat{k}}^{\widehat{j}}
\] 
where $\widehat{2}=1$, $\widehat{1}=2$ and $\widehat{3}=3$. As before, we have that $\omega_{i3}^3=0$. 
Also,  if $j=1$, $k=2$ or $j=2$, $k=1$ then $\omega_{ij}^k=0$.

Now, the frame $\{E_1,E_2,E_3\}$ is chosen in terms of $P$. Thus, Equation \eqref{nablap} will impose conditions on the functions $\omega_{ij}^k$ and $h_{ij}^k$. We divide between types I, II, III and IV.

\subsection{Lagrangian submanifolds of type I}

Equation (\ref{nablap}) for Type I Lagrangian submanifolds yields the following lemma, which was proven in \cite{anarella}.

\begin{lemma}\label{lemmacase1}Let $M$ be a Lagrangian submanifold of the pseudo-nearly Kähler $\Sl\times\Sl$. Suppose that  $A$ and $B$ take type I form in Lemma \ref{propAB} with respect to a $\Delta_1$-orthonormal frame $\{E_1,E_2,E_3\}$. 
    Except for $h_{12}^3$, all the components of the second fundamental form are given by the derivatives of the angle functions $\theta_1,\theta_2$ and $\theta_3$:
\begin{equation}
    E_i(\theta_j)=-\delta_i\delta_jh_{jj}^i,
    \label{anglederi}
\end{equation}
where $\delta_i=g(E_i,E_i)$. Also 
\begin{equation}
   h_{ij}^k\cos(\theta_j-\theta_k)=(\tfrac{1}{\sqrt{6}}\delta_k\varepsilon_{ij}^k-\omega_{ij}^k)\sin(\theta_j-\theta_k),\label{sffc}
\end{equation}
for $j\neq k$.
\end{lemma}
\subsection{Lagrangian submanifolds of type II} The covariant derivative of $P$ in Equation \eqref{nablap} yields the following lemma for type II Lagrangian submanifolds.
\begin{lemma}\label{nablapcase2}
    Let $M$ be a Lagrangian submanifold of the pseudo-nearly Kähler $\Sl\times\Sl$. 
    Suppose that $A$ and $B$ are of type II in Lemma \ref{propAB} with respect to a $\Delta_2$-orthonormal frame $\{E_1,E_2,E_3\}$. Then $h(E_1,E_1)=0$. 
    Moreover, the derivatives of the angles are given by
    \begin{equation}
        \begin{aligned}
            E_1(\theta_1)&=-h_{11}^1=0, \ \ \ \ &&E_2(\theta_1)=-h_{22}^2, \ \ \ \ &&E_3(\theta_1)=-h_{12}^3 \\
            E_1(\theta_2)&=-h_{33}^2, \ \ \ \ &&E_2(\theta_2)=-h_{33}^1, \ \ \ \ &&E_3(\theta_2)=-h_{33}^3.\\
        \end{aligned} \label{derivativesthetatype2}
    \end{equation}
    Furthermore,
    \begin{equation}
         h_{33}^1=-2h_{22}^2, \ \ \ \ \ \ \ \ \ \ h_{33}^2=-2h_{11}^1=0, \ \ \ \ \ \ \ \ \ \ h_{33}^3=-2h_{12}^3.
    \end{equation}
\end{lemma}
\begin{proof}
    Computing Equation \eqref{nablap} with $X=E_2$, $Y=E_1$ and looking on the components of $E_2$ and $JE_2$ we obtain
    \begin{equation*}
        h_{11}^1 \sin2\theta_1 =0, \ \ \ \ h_{11}^1\cos 2\theta_1=0.
    \end{equation*}
    Since sine and cosine never vanish at the same time we get that $h_{11}^1=0$. Same can be done computing Equation \eqref{nablap} with $X=E_1$, $Y=E_1$ and $X=E_3$, $Y=E_1$ and looking in the directions of $E_2$ and $JE_2$ , obtaining $h_{11}^2=h_{11}^3=0$. Looking in the direction of $E_1$ and $JE_1$ on the same equations, we obtain the derivatives of the function $\theta_1$. We derive $E_i(\theta_2)$ by computing Equation \eqref{nablap} with $X=E_i$ and $Y=E_3$.

    We obtain the last statement either from Proposition \ref{minimal}, or from \eqref{derivativesthetatype2} and the fact that $2\theta_1+\theta_2=0$.
\end{proof}
\subsection{Lagrangian submanifolds of type III}
For Lagrangian submanifolds of type III, Equation~\eqref{nablap} gives expressions for all functions $\omega_{ij}^k$, given in the following lemma.
\begin{lemma}\label{nablathirdcase}
    Let $M$ be a Lagrangian submanifold of the pseudo-nearly Kähler $\Sl\times\Sl$. 
    Suppose that $A$ and $B$ take type III form in Lemma \ref{propAB} with respect to a $\Delta_2$-orthonormal frame $\{E_1,E_2,E_3\}$. 
    Then we have $h(E_1,E_1)=0$ and 
    \begin{equation}
        \begin{split}
            h_{12}^3&=\omega_{11}^1=\omega_{11}^3=\omega_{33}^2=0,\\
            \omega_{12}^3&=\frac{\sqrt{2}+(-1)^{k+1 }3 h_{22}^2 }{2 \sqrt{3}},\\
            \omega_{31}^1&=\frac{\sqrt{2}+(-1)^{k +1}12 h_{22}^2 }{2 \sqrt{3}},\\
            \omega_{21}^3&=-\frac{\sqrt{2}+(-1)^{k }6 h_{22}^2 }{2 \sqrt{3}},\\
            \omega_{22}^2&=\frac{(-1)^{k } ( h_{22}^2-3 h_{22}^3)}{\sqrt{3}},\\
            \omega_{33}^1&=\frac{(-1)^{k +1} (4  h_{22}^2-3 h_{22}^3)}{2 \sqrt{3}},\\
            \omega_{22}^3&=\frac{(-1)^{k+1 } (9 h_{22}^1-8  h_{22}^2+6 h_{22}^3)}{6 \sqrt{3}}.
        \end{split}\label{relationsomegatype3}
    \end{equation}
    where $(-1)^k$ with $k\in\{0,1\}$ is the sign of $B$ in Lemma \ref{propAB}.
\end{lemma}
\begin{proof}
    From computing Equation \eqref{nablap} with $X=E_i$, $Y=E_1$, $i=1,2,3$ and looking at the components in the direction of $JE_2$ it follows $h(E_1,E_1)=0$. 
    Now, if we compute Equation \eqref{nablap} with $X=E_i$, $Y=E_1$, $X=E_2$ $Y=E_3$ in the direction on $E_1$ and $JE_2$ respectively we get $h_{12}^3=\omega_{11}^3=0$. 
    The rest of the equations are obtained in the same way.
\end{proof}

\subsection{Lagrangian submanifolds of type IV}
From Lemma \ref{propAB} we know that $2\theta_1+\theta_2=0$ modulo $\pi$. 
Thus, we write $\theta_2=k\pi-2\theta_1$ where $k=0,1$. 
Contrarily to type II Lagrangian submanifolds, here it is necessary to distinguish between $k=0$ and $k=1$, since both $\theta_2$ and $2\theta_2$ appear in the expressions for $A$ and $B$.

From Proposition \ref{minimal} we may assume that
\begin{equation*}
    h^i_{33}=h^i_{22}-h^i_{11}
\end{equation*}
for $i=1,2,3$.
Now, from Equation \eqref{nablap} we obtain the next lemma.
\begin{lemma}\label{nablapfourthcase}
    Given a Lagrangian submanifold of $\Sl\times\Sl$ of type IV we have that
    \begin{equation}
    \begin{split}
        E_1(\theta_1)&=\frac{h_{22}^1-h_{11}^1}{2},\ \ \  E_2(\theta_1)=\frac{h_{11}^2-h_{22}^2}{2},\ \ \    E_3(\theta_1)=\frac{h_{22}^3-h_{11}^3}{2},\\
        E_1(\psi)&= (-1)^{k }2h_{11}^2,\ \ \ E_2(\psi)=(-1)^{k+1 }2h_{22}^1 ,\ \ \ E_3(\psi)=(-1)^{k+1 } 2 h_{12}^3 .
    \end{split}\label{derivativeslambda}
\end{equation}
Moreover, we produce the following expressions for the functions $\omega_{ij}^k$.

\begin{equation}
    \begin{split}
        \omega_{11}^2&=\frac{(-1)^{k }}{2} (h_{11}^1+h_{22}^1) \coth \psi,\\
        \omega_{22}^1&=\frac{(-1)^{k +1}}{2} (h_{11}^2+h_{22}^2) \coth \psi,\\
        \omega_{32}^1&=\frac{(-1)^{k }}{2}  (h_{11}^3+h_{22}^3) \coth \psi-\frac{1}{\sqrt{6}}.
    \end{split}\label{firstomegas}
\end{equation}
Also,
\begin{equation}
    \begin{split}
        \omega_{11}^3&=\frac{ h_{11}^3 \sin 6 \theta_1+(-1)^{k }h_{12}^3 \sinh \psi}{\cos 6 \theta_1-\cosh \psi},\\
        \omega_{12}^3&=\frac{ h_{12}^3 \sin 6 \theta_1+(-1)^{k +1}h_{11}^3 \sinh \psi}{\cos (6\theta_1)-\cosh \psi}+\frac{1}{\sqrt{6}},\\
        \omega_{22}^3&=\frac{ h_{22}^3 \sin 6 \theta_1+(-1)^{k +1}h_{12}^3 \sinh \psi}{\cos 6 \theta_1-\cosh \psi},\\
        \omega_{21}^3&=\frac{ h_{12}^3 \sin 6 \theta_1+(-1)^{k }h_{22}^3 \sinh \psi}{\cos 6 \theta_1-\cosh \psi}-\frac{1}{\sqrt{6}},\\
        \omega_{33}^1&=\frac{(h_{11}^1-h_{22}^1) \sin 6 \theta_1+(-1)^{k }(h_{22}^2-h_{11}^1) \sinh \psi}{\cos 6 \theta_1-\cosh \psi},\\
        \omega_{33}^2&=\frac{(h_{11}^2-h_{22}^2) \sin 6 \theta_1+(-1)^{k +1}(h_{22}^1-h_{11}^1) \sinh \psi}{\cos 6 \theta_1-\cosh \psi}.\\
    \end{split}\label{restomegas}
\end{equation}

\end{lemma}
\begin{proof}
    From computing $(\nabla_{E_1}P)E_1$ and looking at the components in the direction of $E_2$ and $JE_2$ we get the equation
        \begin{equation*}
            \left(
\begin{array}{cc}
 \cos2\theta_1 & \sin2\theta_1 \\
 -\sin2\theta_1 & \cos2\theta_1 \\
\end{array}
\right)\left(
\begin{array}{c}
 \cosh \psi \left( E_1(\psi) -(-1)^{k }2 h_{11}^2\right) \\
  -\sinh \psi \left(2 E_1(\theta_1)+h_{11}^1-h_{22}^1\right) \\
\end{array}
\right)=0.
        \end{equation*}
        From $(\nabla_{E_2}P)E_2$ and $(\nabla_{E_3}P)E_3$ we derive the rest of the equations in Equation \eqref{derivativeslambda} in a similar way.

        Having Equation \eqref{derivativeslambda} we look at the rest of the components of $(\nabla_{E_i}P)E_i$ and we furnish the expressions in \eqref{firstomegas}.

        By computing $(\nabla_{E_i}P)E_j$ with $i\neq j$ we come to linear equations which yield 
     \eqref{restomegas}.
\end{proof}

\section{Extrinsically homogeneous Lagrangian submanifolds} \label{sectionextrhomosub}
In this section we first prove that 
for each case of Lemma \ref{propAB} there is a unique frame $\{E_i\}_i$ with respect to which $P$ takes that particular shape.
Consequently, the associated angle functions, $h_{ij}^k$ and $\omega_{ij}^k$ are constant.
Afterwards, we describe the examples given in Theorem \ref{maintheorem} and provide a classification for each type of Lagrangian submanifold.

\subsection{The uniqueness of the frames}
We consider each case of Lemma \ref{propAB} separately.
\subsubsection{Lagrangian submanifolds of type I}
It is straightforward to check that for type I Lagrangian submanifolds the frame $\{E_1,E_2,E_3\}$ is unique if and only if the functions $\theta_i$ are all different modulo $\pi$. Later on, we will see that if two of them are equal, the submanifold is totally geodesic.
Hence, now we focus on the case where all angles are different modulo $\pi$.
\begin{proposition}\label{frameunique}
    Let $M$ be an extrinsically homogeneous Lagrangian submanifold of the pseudo-nearly Kähler $\Sl\times\Sl$. 
    Suppose that $\{E_1,E_2,E_3\}$ is the unique $\Delta_1$-orthonormal frame such that $A$ and $B$ take type I form in Lemma \ref{propAB}. 
    Then the functions $\theta_i$, $h_{ij}^k$ and $\omega_{ij}^k$ are constant.
\end{proposition}
\begin{proof}
    We have to show that ${\theta_i}_p={\theta_i}_q$ for any two points $p$ and $q$ in $M$. 
    By hypothesis 
    there is a Lie subgroup $H$ of $\Sl\times\Sl\times\Sl$ such that $H$ acts transitively on $M$. 
    Therefore, there exists an isometry $\phi\in H$ such that $\phi(p)=q$. 
    We have that
    \[
    P_p{E_i}_p=\cos 2{\theta_i}_p{E_i}_p+\sin {2\theta_i}_pJ_p{E_i}_p
    \]
    In subsection \ref{seciso} we saw that isometries in $\Sl\times\Sl\times\Sl$ preserve $P$ and $J$. Thus, we apply $\phi$ to both sides:
    \begin{equation*}
        \begin{split}
            P_q \phi_*{E_i}_p=\phi_*P_p{E_i}_p&=\phi_*(\cos 2{\theta_i}_p{E_i}_p+\sin {2\theta_i}_pJ_p{E_i}_p)\\
            &=\cos 2{\theta_i}_p\phi_*{E_i}_p+\sin {2\theta_i}_p\phi_*J_p{E_i}_p\\
            &=\cos 2{\theta_i}_p\phi_*{E_i}_p+\sin {2\theta_i}_pJ_q\phi_*{E_i}_p\\
        \end{split}
    \end{equation*}
    Since $\{E_i\}_i$ is the unique frame with respect to which $A$ and $B$ are diagonal, we have $\phi_*{E_i}_p={E_i}_q$ and ${\theta_i}_q={\theta_i}_p$. 
    
    It follows from $\phi(M)=M$ that $\phi$ preserves $\nabla$ and $h$.
    Thus using a similar argument we get that $\omega_{ij}^k$ and $h_{ij}^k$ are constant.
\end{proof}
\subsubsection{Lagrangian submanifolds of type II}
\begin{proposition}\label{frameuniquecase2}
    Let $M$ be an extrinsically homogeneous Lagrangian submanifold of the pseudo-nearly Kähler $\Sl\times\Sl$. Suppose that $A$ and $B$ take the type II form in Lemma~\ref{propAB} with respect to a $\Delta_2$-orthonormal frame $\{E_1,E_2,E_3\}$.
    If $\theta_1\neq\theta_2$ modulo $\pi$ then the frame is unique up to signs. If instead $\theta_1=\theta_2$  there is a unique frame, up to signs, such that $h_{22}^1=g(h(E_2,E_2),JE_2)=0$. In both cases, the functions $\theta_i$, $h_{ij}^k$ and $\omega_{ij}^k$ are constant.
\end{proposition}
\begin{proof}
    The last statement follows from the uniqueness (even if it is up to sign) of the frame as in the proof of Proposition~\ref{frameunique}.

    Suppose that $\theta_1\neq\theta_2$ and that $\{\tilde{E}_1,\tilde{E}_2,\tilde{E}_3\}$ is a frame on $M$ such that
    \begin{equation*}
        \begin{split}
           P\tilde{E}_1&=\cos 2\tilde{\theta}_1\tilde{E}_1+\sin 2\tilde{\theta}_1J\tilde{E}_1,\\
           P\tilde{E}_2&=\tilde{E}_1+\cos 2\tilde{\theta}_1 \tilde{E}_2-\cot{2\tilde{\theta}_1}J\tilde{E}_1+\sin 2\tilde{\theta}_1J\tilde{E}_2\\
           P\tilde{E}_3&=\cos 2\tilde{\theta}_2\tilde{E}_3+\sin 2\tilde{\theta}_2 J\tilde{E}_3.
        \end{split}
    \end{equation*} 
    Hence, at any point of $M$ the eigenvalues of $A$ and $B$ are $\{\cos 2\theta_1,\cos 2\theta_2\}$ and $\{\sin2\theta_1,\sin2\theta_2\}$, respectively. Moreover, the associated eigenvectors are the same.
    The eigenspace associated to $\cos 2\theta_1$ and $\sin2\theta_1$ is lightlike and the one associated to $\cos 2\theta_2$ and $\sin2\theta_2$ is spacelike, therefore $\tilde{\theta}_1=\theta_1$ and $\tilde{\theta}_2=\theta_2$ modulo $\pi$.
     As the eigenvectors $E_3$  and $\tilde{E}_3$ are both unit length, we derive that $\tilde{E}_3=\pm E_3$. 
     Similarly, we get $\tilde{E}_1=cE_1$ for $c\in\R$ and therefore $\tilde{E}_2=c^{-1}E_2$. Computing $A\tilde{E}_2$ we produce 
    \[
    A\tilde{E}_2=\frac{1}{c}(E_1+\cos2\theta_1E_2)=\frac{1}{c^2}\tilde{E_1} +  \cos2\theta_1 \tilde{E_2}
    \]
    thus $c^2=1$. Since $\{\tilde{E}_i\}$ also has to satisfy the relations in \eqref{tabla} we obtain $\tilde{E_3}=E_3$.

    Suppose now that $\theta_2=\theta_1$ modulo $\pi$.
    This means that the eigenspace associated to $\cos 2\theta_1$ is two-dimensional. 
    Therefore any linear isometry that preserves the eigenspace preserves the form of $A$ and $B$. 
    Let $T$ be the linear isometry defined by $T E_i=\Tilde{E}_i$. 
    After some computations we obtain that $T$ has the form 
\[
T=\begin{pmatrix}
    \varepsilon & -\varepsilon\frac{t^2}{2} & t \\
    0 &\varepsilon & 0\\
    0 & -\varepsilon t & 1 
\end{pmatrix},
\]
for $\varepsilon=\pm1$ and some $t\in\R$.

Computing $h(\tilde{E}_2,\tilde{E}_2)$ and using Lemma \ref{nablapcase2} yields $\tilde{h}_{22}^1=h_{22}^1-2t h_{22}^3$ and $\tilde{h}_{22}^3=h_{22}^3$.  
Suppose that $h_{22}^3=0$, then the Gauss equation (\ref{Gauss}) with $X=E_2,Y=E_3,Z=E_3$ implies $1/6=3/2$, a contradiction. 
Then we can choose $t=h_{22}^1/(2h_{22}^3)$, thus $\Tilde{h}_{22}^1=0$. 
In the same way, we can obtain that it is the unique (up to sign) frame with this condition.
As before, we conclude that $\omega_{ij}^k$ and $h_{ij}^k$ are constant for this frame.
\end{proof}

\subsubsection{Lagrangian submanifolds of type III}
\begin{proposition} \label{constanttype3}
        Let $M$ be an extrinsically homogeneous Lagrangian submanifold of the pseudo-nearly Kähler $\Sl\times\Sl$. 
        Suppose that  $A$ and $B$ take type III form in Lemma~\ref{propAB} with respect to a $\Delta_2$-orthogonal frame $\{E_1,E_2,E_3\}$. 
        Then the frame is unique and the functions $\theta_i$, $h_{ij}^k$ and $\omega_{ij}^k$ are constant.
\end{proposition}
\begin{proof}
    Suppose that $\{\tilde{E}_i\}_i$ is another $\Delta_2$-orthonormal frame wiht respect to which $A$ and $B$ take type III form in Lemma~\ref{propAB}. We denote by $T$ the linear isometry given by $TE_i=\tilde{E}_i$.
         We write 
    \[
    T=\begin{pmatrix}
        t_{11} & t_{12} & t_{13}\\
        t_{21} & t_{22} & t_{23} \\
        t_{31} & t_{32} & t_{33}
     \end{pmatrix}.
    \]
    First notice that $E_1$ spans the unique eigenspace of $A$ and $B$, thus $t_{21}=t_{31}=0$. 
    Computing $g(TE_1,TE_2)$ we furnish $t_{22}t_{11}=1$. 
    In the same way, computing $g(TE_1,TE_3)$ and $g(TE_3,TE_3)$ we come to $t_{23}=0$ and $t_{33}=\varepsilon=\pm1$. 
    Computing $ATE_2=-\tfrac{1}{2} TE_2 +TE_3$ we get that $t_{13}=t_{32}$ and $t_{11}=\varepsilon$. 
    Computing $g(E_2,E_3)$ and $g(E_2,E_2)$ we obtain $t_{13}=0$ and $t_{12}=0$. 
    By asking $JG(TE_1,TE_2)=\sqrt{\tfrac{2}{3}}TE_3$ we see that $\varepsilon=1$. 

    As in the previous propositions, the last statement follows from the uniqueness of the frame.
\end{proof}
\subsubsection{Lagrangian submanifolds of type IV}
\begin{proposition}\label{frameuniquecase4}
        Let $M$ be an extrinsically homogeneous Lagrangian submanifold of the pseudo-nearly Kähler $\Sl\times\Sl$. Suppose that $A$ and $B$ take type IV form in Lemma \ref{propAB}  with respect to $\Delta_3$-orthonormal frame $\{E_1,E_2,E_3\}$. Then the frame is unique and the functions $\theta_i$, $\psi$, $h_{ij}^k$ and $\omega_{ij}^k$ are constant.
\end{proposition}
\begin{proof}
    In order to simplify the proof we write $A$ and $B$ as 
    \[
    A=\begin{pmatrix}
        \alpha & \beta & 0\\
        -\beta & \alpha & 0\\
        0 & 0 & \cos2\theta_2\\
    \end{pmatrix}, \ \ \ \ B=\begin{pmatrix}
        \gamma & \delta & 0\\
        -\delta & \gamma & 0\\
        0 & 0 & \sin2\theta_2\\
    \end{pmatrix}, 
    \]
    with respect to $\{E_i\}_i$.
    
    Suppose there exists a $\Delta_3$-orthonormal frame $\{\tilde{E}_1,\tilde{E}_2,\tilde{E}_3\}$ and functions $\tilde{\alpha}$, $\tilde{\beta}$, $\tilde{\gamma}$, $\tilde{\delta}$ and $\tilde{\theta}_2$ such that $A$ and $B$ take the form
    \[
    A=\begin{pmatrix}
        \tilde{\alpha} & \tilde{\beta} & 0\\
        -\tilde{\beta} & \tilde{\alpha} & 0\\
        0 & 0 & \cos2\tilde{\theta}_2\\
    \end{pmatrix}, \ \ \ \ B=\begin{pmatrix}
        \tilde{\gamma} & \tilde{\delta} & 0\\
        -\tilde{\delta} & \tilde{\gamma} & 0\\
        0 & 0 & \sin2\tilde{\theta}_2\\
    \end{pmatrix}, 
    \]
    with respect to $\{\tilde{E}_i\}_i$.

Since $\cos2\theta_2$ and $\sin{2\theta_2}$ are the only eigenvalues of $A$ and $B$, $\tilde{\theta}_2=\theta_2$ modulo $\pi$ and $\tilde{E_3}=\varepsilon E_3$ with $\varepsilon=\pm1$.
We denote by $T$ the linear isometry defined by $TE_i=\tilde{E}_i$. 
Given that $\{\tilde{E}_i\}$ is an $\Delta_3$-orthonormal frame, we may assume that $\tilde{E}_1$ and $\tilde{E}_2$ do not have components in the direction of $E_3$. 
Hence, we can write $T$ as 
\[
T=\begin{pmatrix}
    \cosh t & \sinh t & 0\\
    \sinh t & \cosh t & 0\\
    0&0 & \varepsilon\\
\end{pmatrix}.
\]
with $t\in\R$.
Requiring $ATE_1=(\tilde{\alpha} TE_1-\tilde{\beta}TE_2)$ and $ATE_2=(\tilde{\beta} TE_1+\tilde{\alpha}TE_2)$ we obtain
\begin{equation*}
    \begin{split}
        \alpha \cosh t +\beta \sinh t &=\tilde{\alpha} \cosh t -\tilde{\beta} \sinh t ,\\
         -\beta \cosh t +\alpha \sinh t&=-\tilde{\beta} \cosh t +\tilde{\alpha} \sinh t,  \\
         \beta \cosh t +\alpha \sinh t  &=\tilde{\beta} \cosh t +\tilde{\alpha} \sinh t, \\
         \alpha \cosh t -\beta \sinh t&= \tilde{\alpha} \cosh t +\tilde{\beta} \sinh t.
    \end{split}
\end{equation*}
Combining these equations we get
\begin{equation*}
    \begin{split}
        (\alpha-\tilde{\alpha}) \cosh t=0\\
        (\beta-\tilde{\beta}) \cosh t=0\\
    \end{split}
\end{equation*}
therefore $\tilde{\alpha}=\alpha$ and $\tilde{\beta}=\beta$. 
We may use the same argument to deduce that $\tilde{\delta}=\delta$ and $\tilde{\gamma}=\gamma$. 
We compute again $ATE_1=(\tilde{\alpha} TE_1-\tilde{\beta}TE_2)$ and we derive $t=0$. 
From $JG(\tilde{E}_1,\tilde{E}_2)=\sqrt{\tfrac{2}{3}}\tilde{E}_3$ it follows that $\varepsilon=1$.

In a similar way as in the proofs of propositions \ref{frameunique}-\ref{constanttype3} we obtain that $\alpha$, $\beta$, $\gamma$ and $\delta$ are constant. Computing $\tfrac{\alpha}{\delta}$ and $\tfrac{\alpha}{\beta}$ we obtain that $\psi$, $\theta_1$ and $\theta_2$ are constant as well. Finally using the uniqueness of the frame we get that $h_{ij}^k$ and $\omega_{ij}^k$ are constant.

\end{proof}

\subsection{Extrinsically homogeneous Lagrangian submanifolds of type I}
The following proposition proven in \cite{anarella} gives us a characterization of totally geodesic Lagrangian submanifolds of type~I:
\begin{proposition}\label{twoanglesequal}
    Let $M$ be a Lagrangian submanifold of $\Sl\times\Sl$ of type I in Lemma \ref{propAB}. If two angles are equal modulo $\pi$, then $M$ is totally geodesic.
    \end{proposition}

    Moreover, in the same paper all the totally geodesic Lagrangian submanifolds are classified up to congruence by the following theorem, which we rewrite to fit better in this article.
\begin{theorem}\label{totgeo}
Let $f:(M,g)\to\Sl\times\Sl$  be a totally geodesic Lagrangian submanifold of the pseudo-nearly Kähler $\Sl\times\Sl$. Then $f(M)$ is congruent to an open subset of the following extrinsically homogeneous Lagrangian embeddings: 
\begin{equation*}
    \begin{tblr}{c c c}
    \hline
    (M,g) & f & H \\
        \hline
    (\Sl, \frac{2}{3}\li,\ri) & u\mapsto (u,u) & \{ (u,u,\id_2):u\in\Sl\} \\
    (\Sl,g^+_{\kappa,\tau})  & u\mapsto (u,\ii u\ii) & \{(u,\ii u\ii,\id_2):u\in \Sl\}\\
    (\Sl,g^-_{\kappa,\tau}) & u\mapsto (u,-\kk u\kk) & \{(u,-\kk u\kk,\id_2):u\in \Sl\}\\
    \hline
\end{tblr}
\end{equation*}
where $H$ acts transitively on $M$ with null isotropy, and $g_{\kappa,\tau}^+$, $g_{\kappa,\tau}^-$ are Berger-like metrics on $\Sl$ stretched in spacelike and timelike directions, respectively. 
\end{theorem}
This theorem implies that any Lagrangian submanifold of type I with two equal angle functions modulo $\pi$ is extrinsically homogeneous. Therefore, to complete the classification of extrinsically homogeneous Lagrangian submanifolds we assume that the submanifold is not totally geodesic and that all angle functions are different modulo $\pi$.

\begin{proposition}\label{typeIcurvature} 
    Let $M$ be a non-totally geodesic extrinsically homogeneous Lagrangian submanifold of the pseudo-nearly Kähler $\Sl\times\Sl$ of type I. 
    Let $\theta_i$, $i=1,2,3$ be the angle functions associated to the $\Delta_1$-orthonormal frame with respect ot which $A$ and $B$ are diagonal. 
    Then $(\theta_1,\theta_2,\theta_3)$ is a permutation of $(0,\pi/3,2\pi/3)$ and the manifold $M$ has constant sectional curvature. 
    Moreover, the sectional curvature is either equal to $0$ or to $-\tfrac{3}{8}$.
\end{proposition}

\begin{proof}
    Lagrangian submanifolds of type I are essentially an analogue of Lagrangian submanifolds of $\Ss^3\times\Ss^3$.
    In \cite{constantangles} the authors proved  for $\Ss^3\times\Ss^3$ that the angle functions of non-totally geodesic Lagrangian submanifolds are constant  and a permutation of $(0,\tfrac{\pi}{3},\tfrac{2\pi}{3})$. The same argument works for $\Sl\times\Sl$. 

    By Lemma~\ref{isometrieswithAB} we may assume that $(\theta_1,\theta_2,\theta_3)=(0,\pi/3,2\pi/3)$.  
    From Lemma \ref{lemmacase1} we know that all the functions $h_{ij}^k$  are equal to zero, except for $h_{12}^3$, which from Proposition \ref{frameunique} we know is constant.
    Then the Codazzi equation~\eqref{Codazzi} with $X=E_1$, $Y=E_2$, $Z=E_2$  yields that $h_{12}^3$ is either equal to $\frac{1}{2\sqrt{2}}$ or to $-\frac{1}{\sqrt{2}}$.
    Both cases imply that the sectional curvature is constant. 
    In the former case the sectional curvature is equal to $-\tfrac{3}{8}$ and in the latter case the sectional curvature is equal to~$0$.
       \end{proof}

\begin{example}\label{pslexample}
Let $f:\Sl\to\Sl\times\Sl$ be the isometric immersion given by $u\mapsto(\ii u\ii u^{-1},\jj u \jj u^{-1})$ and let $\{X_1,X_2,X_3\}$ be the frame on $\Sl$ given in \eqref{frameXi}. We may compute
\begin{equation*}
\begin{split}
    f_*(X_1)&
    =(0,\jj u\jj u^{-1}(-2u\ii u^{-1})),\\
    f_*(X_2)&
    =(\ii u\ii u^{-1}(-2u\jj u^{-1}),0),\\
    f_*(X_3)&=(\ii u \ii u^{-1}(-2u\kk u^{-1}),\jj u\jj u^{-1}(-2u\kk u^{-1})).\\
\end{split}
\end{equation*}
It follows from the definition of $J$ in \eqref{defJ} that
\begin{equation*}
    \begin{split}
        Jf_*(X_1)&=\frac{1}{\sqrt{3}}(\ii u \ii u^{-1}(4u\ii u^{-1}),\jj u \jj u^{-1}(2u\ii u^{-1})),\\
        Jf_*(X_2)&=\frac{1}{\sqrt{3}}(\ii u \ii u^{-1}(-2u\jj u^{-1}),\jj u \jj u^{-1}(-4u\jj u^{-1})),\\
        Jf_*(X_3)&=\frac{1}{\sqrt{3}}(\ii u \ii u^{-1}(2u\kk u^{-1}),\jj u \jj u^{-1}(-2u\kk u^{-1})).\\
    \end{split}
\end{equation*}
We can easily check that $f$ is a Lagrangian immersion by computing $g(Jf_*(X_i),f_*(X_j))=0$ for $i,j=1,2,3$. Moreover, we have 
\begin{equation*}
    \begin{split}
        Pf_*(X_1)&=(\ii u\ii u^{-1}(-2u\ii u^{-1}),0)=-\frac{1}{2}f_*(X_1)-\frac{\sqrt{3}}{2}Jf_*(X_1), \\
        Pf_*(X_2)&=(0,\jj u\jj u^{-1}(-2u\jj u^{-1}))=-\frac{1}{2}f_*(X_2)+\frac{\sqrt{3}}{2}Jf_*(X_2), \\
        Pf_*(X_3)&=(\ii u \ii u^{-1}(-2u\kk u^{-1}),\jj u\jj u^{-1}(-2u\kk u^{-1}))=f_*(X_3). \\
    \end{split}
\end{equation*}
Let $H$ be the subgroup of $\isoo(\Sl\times\Sl)$ given by $\{(\ii u \ii , \jj u \jj, u):u\in\Sl\}\cong\Sl$. Then $f(\Sl)=H\cdot (\id_2,\id_2)$. Notice that $H$ acts on $f(\Sl)$ with isotropy~$\Z_2$. Hence, the embedding $\Psl\to\Sl\times\Sl:[u]\mapsto(\ii u\ii u^{-1},\jj u \jj u^{-1})$ is congruent to $f$.
\end{example}

\begin{proposition}\label{psl}
Any extrinsically homogeneous non-totally geodesic Lagrangian submanifold of the pseudo-nearly Kähler $\Sl\times\Sl$ of type I with $h_{12}^3=\frac{1}{2\sqrt{2}}$ is congruent to an open subset of the image of $\mathrm{PSL}(2,\R)\to\Sl\times\Sl:[u]\mapsto(\ii u\ii u^{-1},\jj u\jj u^{-1})$.
\end{proposition}
\begin{proof}
  Let $f\colon M\to\Sl\times\Sl$ be a non-totally geodesic extrinsically homogeneous Lagrangian immersion of type I.
  Let $\{E_1,E_2,E_3\}$ be the frame on $M$  such that $JG(E_1,E_2)=\sqrt{\tfrac{2}{3}}E_3$ with angle functions given by $(\theta_1,\theta_2,\theta_3)=(0,\pi/3,2\pi/3)$. Moreover, assume that $h_{12}^3=g(h(E_1,E_2),E_3)=\frac{1}{2\sqrt{2}}$. We have
\[
PE_1=E_1,\ \ \  PE_2=-\tfrac{1}{2}E_2+\tfrac{\sqrt{3}}{2}JE_2,\ \ PE_3=-\tfrac{1}{2}E_3-\tfrac{\sqrt{3}}{2}JE_3.
\]
Hence, according to Equation \eqref{prodQ}, we have
\begin{equation}
    QE_1=\sqrt{3}JE_1,\ \ \ QE_2=-E_2, \ \ \ QE_3=E_3.\label{eqQ}
\end{equation}
From Proposition \ref{typeIcurvature} it follows that $M$ has constant sectional curvature $-\tfrac{3}{8}$. 
Thus $M$ is locally isometric to $(\Sl,\tfrac{8}{3} \li,\ri)$ (see \cite{oneill}), where $\li,\ri$ is the metric given in \eqref{prodsl2}. Then we may identify 
\[
E_1=\sqrt{\tfrac{3}{8}}X_3,\ \ \ E_2=\sqrt{\tfrac{3}{8}}X_2, \ \ \ E_3=\sqrt{\tfrac{3}{8}}X_1.
\]
where $\{X_1,X_2,X_3\}$ is the frame on $\Sl$ given in \eqref{frameXi}.
Now we write the immersion $f(u)=(p(u),q(u))$ and $f_*(E_i)_u=(D_{E_i}f)_u=(p(u)\alpha_i(u),q(u)\beta_i(u))$ where $\alpha_i(u),\beta_i(u)\in\mathfrak{sl}(2,\R)$. 
By Equation~\eqref{eqQ} we have $\alpha_1=\beta_1$, $\beta_2=0$ and $\alpha_3=0$.
We know from \eqref{relprodkal} that $\nabla^E_{E_1}E_1=\nabla^E_{E_2}E_2=\nabla^E_{E_3}E_3=0$ and
\begin{equation}
    \begin{split}
        \nabla^E_{E_1}E_2&=\nabla^E_{E_1}E_3=0,\\
        \nabla^E_{E_2}E_1&=-\sqrt{\tfrac{3}{2}}E_3=-\sqrt{\tfrac{3}{2}}(0,q\beta_3),\\
        \nabla^E_{E_2}E_3&=-\tfrac{1}{2}\sqrt{\tfrac{3}{2}}(E_1+QE_1)=-\sqrt{\tfrac{3}{2}}(0,q\alpha_1),\\
        \nabla^E_{E_3}E_1&=\sqrt{\tfrac{3}{2}}E_2=\sqrt{\tfrac{3}{2}}(p\alpha_2,0),\\
        \nabla^E_{E_3}E_2&=\tfrac{1}{2}\sqrt{\tfrac{3}{2}}(E_1-QE_1)=\sqrt{\tfrac{3}{2}}(p\alpha_1,0).\\
    \end{split}\label{eqconprod}
\end{equation}
Throughout this proof, we will denote by $\li,\ri_\times$ the product metric associated to the metric $\li,\ri$ on $\Sl$ given in \eqref{prodsl2}. 
By Equation \ref{prodmetric} $E_1,E_2,E_3$ are orthogonal with respect to the product metric $\li,\ri_\times$ and their lengths are
\[
\li E_1,E_1\ri_\times=-3, \ \ \ \li E_2,E_2\ri_\times=\li E_3,E_3\ri_\times=\tfrac{3}{2}.
\]
This implies that
\begin{equation}
    \li \alpha_1,\alpha_1 \ri =-\tfrac{3}{2}, \ \ \ \li \alpha_2,\alpha_2\ri=\li \beta_3,\beta_3\ri=\tfrac{3}{2}    \label{lenghtalphaspsl}
\end{equation}
On the one hand, Equation \eqref{connectionr8} yields
\[
D_{E_i}D_{E_j}f=\nabla^E_{E_i}E_j+\tfrac{1}{2}\li E_i,E_j\ri(p,q)+\tfrac{1}{2}\li E_i,QE_j\ri (-p,q),
\]
and on the other hand, by Equation \eqref{prodinslf}, we furnsih
\begin{equation*}
    \begin{split}
        D_{E_i}D_{E_j}f&=(p\alpha_i\alpha_j+pE_i(\alpha_j),q\beta_i\beta_j+qE_i(\beta_j))\\
        &=(p(\alpha_i\times\alpha_j)+\li \alpha_i,\alpha_j\ri p+pE_i(\alpha_j),q(\beta_i\times\beta_j)+\li \beta_i,\beta_j\ri q+qE_i(\beta_j)).
    \end{split}
\end{equation*}
Therefore
\[
\nabla^E_{E_i}E_j=(p\alpha_i\times\alpha_j+pE_i(\alpha_j),q\beta_i\times\beta_j+qE_i(\beta_j)),
\]
where $E_i(\alpha)=d\alpha(E_i)$ thinking of $\alpha$ as a map from $\Sl$ into $\slf$.
Hence, using (\ref{eqconprod}) we obtain
\[
\alpha_1\times\alpha_2=-\sqrt{\tfrac{3}{2}}\beta_3
\]
and also
\begin{equation*}
    \begin{aligned}
        E_1(\alpha_1)&=0,  &&E_2(\alpha_1)=-\sqrt{\tfrac{3}{2}}\beta_3,  &&E_3(\alpha_1)=\sqrt{\tfrac{3}{2}}\alpha_2,\\
        E_1(\alpha_2)&= \sqrt{\tfrac{3}{2}}\beta_3,  &&E_2(\alpha_2)=0,  &&E_3(\alpha_2)=\sqrt{\tfrac{3}{2}}\alpha_1,\\
        E_1(\beta_3)&=-\sqrt{\tfrac{3}{2}}\alpha_2,  &&E_2(\beta_3) =-\sqrt{\tfrac{3}{2}}\alpha_1, &&E_3(\beta_3)=0.\\ 
    \end{aligned}
\end{equation*}
In terms of the vector fields $X_i$ this translates into the  following differential equations:
\begin{equation}
    \begin{aligned}
        X_1(\alpha_1)&=2\alpha_2, \ \ \ &&X_2(\alpha_1)=-2\beta_3, \ \ \ &&X_3(\alpha_1)=0,\\
        X_1(\alpha_2)&= 2\alpha_1, \ \ \ &&X_2(\alpha_2)=0, \ \ \ &&X_3(\alpha_2)=2\beta_3,\\
        X_1(\beta_3)&=0, \ \ \ &&X_2(\beta_3) =-2\alpha_1, \ \ \ &&X_3(\beta_3)=-2\alpha_2.\\ 
    \end{aligned}\label{systemal1al2be3}
\end{equation}
From~\eqref{lenghtalphaspsl} and Lemma \ref{lemmasl2requaltoso21} we know that there exists $c\in\mathrm{SL}^\pm(2,\R)$ such that 
\begin{equation}
    \alpha_1(\id_2)=-\sqrt{\tfrac{3}{2}}c\kk c^{-1}, \ \ \ \alpha_2(\id_2)=-\sqrt{\tfrac{3}{2}}c\jj c^{-1}, \ \ \ \beta_3(\id_2)=-\sqrt{\tfrac{3}{2}}c\ii c^{-1}. \label{initialconditionspsl}
\end{equation}
Therefore, as the solution of the system (\ref{systemal1al2be3}) with initial conditions \eqref{initialconditionspsl} is unique, we have that
\[
\alpha_1(u)=-\sqrt{\tfrac{3}{2}}cu\kk u^{-1}c^{-1}, \ \ \ \alpha_2(u)=-\sqrt{\tfrac{3}{2}}u\jj u^{-1}c^{-1}, \ \ \ \beta_3(u)=-\sqrt{\tfrac{3}{2}}cu\ii u^{-1}c^{-1}.
\]
We can check easily that they satisfy the equations in (\ref{systemal1al2be3}). By the homogeneity of $\Sl\times\Sl$ we can take initial conditions $f(\id_2)=(\id_2,\id_2)$. Applying the isometry of $\Sl\times\Sl$ given by $(p,q)\mapsto(cpc^{-1},cqc^{-1})$ we may assume that
\[
\alpha_1(u)=-\sqrt{\tfrac{3}{2}}u\kk u^{-1}, \ \ \ \alpha_2(u)=-\sqrt{\tfrac{3}{2}}u\jj u^{-1}, \ \ \ \beta_3(u)=-\sqrt{\tfrac{3}{2}}cu\ii u^{-1}.
\]
Then, the immersion $f=(p,q)$ given by $p=\ii u\ii u^{-1}$, $q=\jj u\jj u^{-1}$ is the unique solution of the differential equation $D_{E_i}p=p\alpha_i$, $D_{E_i}q=q\beta_i$, with $i=1,2,3$.
\end{proof}
\begin{example}\label{toriexample}

Let $H$ be the Lie subgroup of $\Sl\times\Sl\times\Sl$ given by
\[H=\{(e^{v\ii},e^{w\jj},e^{u\kk}):v,w\in \R,\ u\in [0,2\pi)\}\cong\R^2\times \Ss^1.\]
The subgroup $H$ acts transitively on the submanifold  $f:\R^3_1/\mathbb{Z}\to\Sl\times\Sl$ given by 
    \[
    f(u,v,w)=(e^{v\ii}e^{-u\kk},e^{w \jj}e^{-u\kk}).
    \]
    Moreover, the isotropy of $H$ is trivial.

    The derivatives of $f$ are given by
    \begin{equation*}
        \begin{split}
            f_u=(-e^{v\ii} e^{-u \kk}\kk,-e^{w\jj} e^{-u\kk}\kk),\ \ \ \ \ \ f_v=(e^{v\ii} e^{-u \kk} e^{u \kk}\ii  e^{-u \kk},0),\ \ \ \ \ \ f_w=(0,e^{w\jj} e^{-u \kk}e^{u \kk}\jj e^{-u \kk}).
        \end{split}
    \end{equation*}
    Applying the almost complex structure yields the following expressions:
    \begin{equation*}
        \begin{split}
            Jf_u&=\frac{1}{\sqrt{3}}(e^{v\ii} e^{-u \kk}\kk,-e^{w\jj} e^{-u\kk}\kk),\\
            Jf_v&=\frac{1}{\sqrt{3}}( e^{v\ii} e^{-u \kk}  e^{u \kk}\ii e^{-u \kk},2 e^{v\ii} e^{-u \kk}  e^{u \kk}\ii e^{-u \kk}),\\
            Jf_w&=\frac{1}{\sqrt{3}}(-2e^{w\jj} e^{-u \kk}  e^{u \kk}\jj e^{-u \kk}, -e^{w\jj} e^{-u \kk}  e^{u \kk}\jj e^{-u \kk}).\\
        \end{split}
    \end{equation*}
    We can easily check that 
    \[g(Jf_u,f_v)=g(Jf_u,f_w)=g(Jf_v,f_w)=0,\]
    which shows that this submanifold is Lagrangian.
    After applying the tensor $P$ we obtain
    \begin{equation*}
        \begin{split}
            Pf_u&=(-e^{v\ii} e^{-u \kk}\kk,-e^{w\jj} e^{-u\kk}\kk)=f_u,\\
            Pf_v&=(0,e^{w\jj} e^{-u \kk} e^{u \kk}\ii  e^{-u \kk})=-\frac{1}{2}f_v+\frac{\sqrt{3}}{2}Jf_v,\\
            Pf_w&=(e^{w\jj} e^{-u \kk}e^{u \kk}\jj e^{-u \kk},0)=-\frac{1}{2}f_w-\frac{\sqrt{3}}{2}Jf_w.\\
        \end{split}
    \end{equation*}
    Thus, $f$ is a flat, extrinsically homogeneous Lagrangian submanifold of $\Sl\times\Sl$ of type~I with constant angles $(\theta_1,\theta_2,\theta_3)=(0,\tfrac{\pi}{3},\tfrac{2\pi}{3})$.
\end{example}
\begin{proposition}\label{tori}
Any extrinsically homogeneous non-totally geodesic Lagrangian submanifold of type I of the nearly Kähler $\Sl\times\Sl$ with $h_{12}^3=-\frac{1}{\sqrt{2}}$ is congruent to an open subset of the image of the Lagrangian embedding $f:\R^3_1/\Z\to\Sl\times\Sl$ given by
\[
f(u,v,w)=\left(e^{v \ii}e^{-u \kk}, e^{w\jj }e^{-u \kk}\right).
\]
\end{proposition}
\begin{proof}
Notice first that all the coefficients of the connection and second fundamental form vanish except for $h_{12}^3=g(h(E_1,E_2),JE_3)=-\frac{1}{\sqrt{2}}$. We also know that the angle functions are given by $(2\theta_1,2\theta_2,2\theta_3)=(0,\frac{2\pi}{3},\frac{4\pi}{3})$. So we can find a local frame such that $JG(E_1,E_3)=\sqrt{\frac{2}{3}}E_3$ and 
\[
PE_1=E_1, \ \ \ \ PE_2=-\tfrac{1}{2}E_2+\tfrac{\sqrt{3}}{2}JE_2, \ \ \ \ PE_3=-\tfrac{1}{2}E_3-\tfrac{\sqrt{3}}{2}JE_3.
\]
From the relation between $Q$ and $P$ in \eqref{prodQ} it follows that
\begin{equation}
QE_1=\sqrt{3}JE_1, \ \ \ \ QE_2=-E_2, \ \ \ \, QE_3=E_3.    \label{Qstori}
\end{equation}
Using that $h_{ii}^j=0$, $h_{12}^3=-\tfrac{1}{\sqrt{2}}$ and Equation \eqref{sffc}, we deduce that $[E_i,E_j]=0$ for $i,j=1,2,3$. Then we write $E_1=f_u$, $E_2=f_v$, $E_3=f_w$ for $u,v,w$ local coordinates. Thus, Equation \eqref{Qstori} implies that 
\begin{equation}
    p_w=0, \ \ \ \ \ q_v=0,  \ \ \ \ \ q_u=p_u. \label{relationsderivativestori}
\end{equation} 
Moreover, we have
$\nabla^E_{f_u}f_u=\nabla^E_{f_v}f_v=\nabla^E_{f_w}f_w=0$ and 
\begin{equation*}
    \begin{split}
        \nabla^E_{f_u}f_v&=\nabla^E_{f_v}f_u=-\tfrac{1}{2}\sqrt{\tfrac{3}{2}}f_w-\tfrac{3}{2\sqrt{2}}Jf_w,\\
        \nabla^E_{f_u}f_w&=\nabla^E_{f_w}f_u=\tfrac{1}{2}\sqrt{\tfrac{3}{2}}f_v-\tfrac{3}{2\sqrt{2}}Jf_v,\\
        \nabla^E_{f_v}f_w&=\nabla^E_{f_w}f_v=0.\\
    \end{split}
\end{equation*}
From the relation between the Euclidean metric with the nearly Kähler metric in \eqref{prodmetric} we know that $f_u,f_v,f_w$ are also orthogonal with respect to the induced Euclidean product metric. Furthermore, their inner products are given by
\begin{equation*}
\li f_u,f_u\ri=-3, \ \ \ \ \li f_v,f_v\ri=\li f_w,f_w\ri=\tfrac{3}{2}.    \label{metricsftori}
\end{equation*}
Furthermore,
\begin{equation*}
\begin{aligned}
    \li f_u,Qf_u\ri&=0, &&\li f_u, Qf_v\ri=0, &&\li f_u,Qf_w\ri=0,\\
    \li f_v,Qf_w\ri&=0,  &&\li f_v,Qf_v\ri=-\tfrac{3}{2},  &&\li f_w,Qf_w\ri=\tfrac{3}{2}.\label{metricsqtori}
\end{aligned}
\end{equation*}
Thus, using \eqref{relationsderivativestori} we get that $p_u$, $p_v$ and $q_w$ are orthogonal and
\begin{equation}
    \li p_u,p_u\ri=-\tfrac{3}{2}, \ \ \ \ \li p_v,p_v\ri=\li p_w,p_w\ri=\tfrac{3}{2}. \label{lenghtsputori}   
\end{equation}

From the expression for the Euclidean connection $D$ of $\R^8_4$ given in Equation \eqref{connectionr8} we obtain
\begin{equation*}
    \begin{aligned}
        f_{uu}&=-\tfrac{3}{2}f,  &&f_{uv}=-\tfrac{1}{2}\sqrt{\tfrac{3}{2}}f_w-\tfrac{3}{2\sqrt{2}}Jf_w,\\
        f_{uw}&=\tfrac{1}{2}\sqrt{\tfrac{3}{2}}f_v-\tfrac{3}{2\sqrt{2}}Jf_v, &&f_{vv}=\tfrac{3}{4}f-\tfrac{3}{4}Qf,\\
        f_{vw}&=0, &&f_{ww}=\tfrac{3}{4}f+\tfrac{3}{4}Qf.\\
    \end{aligned}
\end{equation*}
Hence we produce differential equations for $p$ and $q$:
\begin{equation}
    \begin{aligned}
        p_{uu}&=-\tfrac{3}{2}p, && p_{vv}=\tfrac{3}{2}p,  &&p_{uv}=\sqrt{\tfrac{3}{2}}pq^{-1}q_w, \\
        q_{uu}&=-\tfrac{3}{2}q, && q_{ww}=\tfrac{3}{2}q,  && q_{uw}=-\sqrt{\tfrac{3}{2}}qp^{-1}p_v. \\
    \end{aligned}\label{difeqsystemtori}
\end{equation}
By applying an isometry of the type $(p,q)\mapsto(ap,bq)$ we may assume that $p(0)=\id_2$, $q(0)=\id_2$. Now, because of Equation \eqref{lenghtsputori}, there exists a matrix $c\in\mathrm{SL}^\pm(2,\R)$ such that 
\[
p_u(0)=\sqrt{\tfrac{3}{2}}c\kk c^{-1}, \ \ \ p_v(0)=\sqrt{\tfrac{3}{2}}c\ii c^{-1}, \ \ \ q_u(0)=\sqrt{\tfrac{3}{2}}c\kk c^{-1}, \ \ \ q_w(0)=\sqrt{\tfrac{3}{2}}c\jj c^{-1}.
\]
Applying the isometry $(p,q)\mapsto (cpc^{-1},cqc^{-1})$ we obtain that any solution of the system \eqref{difeqsystemtori} is congruent to the map $f(u,v,w)=(e^{\sqrt{\tfrac{3}{2}}v\ii}e^{-\sqrt{\tfrac{3}{2}}u\kk},e^{\sqrt{\tfrac{3}{2}}w\jj}e^{-\sqrt{\tfrac{3}{2}}u\kk})$. Then, changing the coordinates by $u\to\sqrt{\tfrac{2}{3}}u $, $v\to\sqrt{\tfrac{2}{3}}v $ and $w\to\sqrt{\tfrac{2}{3}}w$, we get the map $(u,v,w)\mapsto (e^{v\ii}e^{-u\kk},e^{w\jj}e^{-u\kk})$.
\end{proof}

\subsection{Extrinsically homogeneous Lagrangian submanifolds of type II}
\begin{proposition}\label{propositiontype2twotypes}
    Let $M$ be an extrinsically homogeneous Lagrangian submanifold of the pseudo-nearly Kähler $\Sl\times\Sl$. Suppose that $A$ and $B$ take type II form in Lemma \ref{propAB} with respect to a $\Delta_2$-orthonormal frame $\{E_1,E_3,E_3\}$. 
    Then $\theta_1=\theta_2=\pi/3$ modulo $\pi$. 
    Also, all the components of the second fundamental form and of the connection are equal to zero except for $h_{22}^3$, $\omega_{12}^3$, $\omega_{21}^3$, $\omega_{31}^1$ and $\omega_{22}^3$. 
    There are only two possibilities for the value of these constants: 
    \begin{enumerate}
        \item $h_{22}^3=-\frac{\sqrt{2}}{3},\ \ \ \omega_{12}^3=-\sqrt{\frac{3}{2}}, \ \ \ \omega_{21}^3=-\sqrt{\frac{3}{2}}, \ \ \ \omega_{31}^1=0, \ \ \  \omega_{33}^2=0,$ \label{typeIIfirst}
        \item $h_{22}^3=\frac{2 \sqrt{2}}{3},\ \ \ \omega_{12}^3=\sqrt{\frac{3}{2}}, \ \ \ \omega_{21}^3=\sqrt{\frac{3}{2}}, \ \ \ \omega_{31}^1=\sqrt{\frac{3}{2}}.$\label{typeIIsecond}
    \end{enumerate}
    Moreover, in both cases the sectional curvature is constant and equal to $-\tfrac{3}{2}$.
\end{proposition}
\begin{proof}
    As indicated in Proposition \ref{frameuniquecase2}, we have to distinguish between two cases: when $\theta_1=\theta_2$ and when $\theta_1\neq\theta_2$.

    Suppose first that $\theta_1\neq\theta_2$. By Proposition \ref{frameuniquecase2} both angles and the functions $h_{ij}^k$, $\omega_{ij}^k$ are constant.
    It follows from computing the Codazzi equation \eqref{Codazzi} with $X=E_3$, $Y=E_2$ and $Z=E_3$ that $\sin (2 (\theta_1-\theta_2))=0$, hence $\theta_1$ and $\theta_2$ are equal modulo $\pi/2$. 
    Recall that for type II submanifolds $\sin 2\theta_1$ is different from zero, $2\theta_1+\theta_2=0$ modulo $\pi$ and  the angles are different modulo $\pi$. 
    Therefore $\theta_1=\tfrac{\pi}{6}$, $\theta_2=\tfrac{2\pi}{3}$ or $\theta_1=\tfrac{5\pi}{6}$, $\theta_2=\tfrac{\pi}{3}$. 
    From Lemma \ref{isometrieswithAB} we know that these two cases are congruent via the isometry $\Psi_{0,1}$ given in \eqref{isoslsl}. It follows from Equation \eqref{nablap} that 
        \begin{equation*}
            \omega_{11}^1=\omega_{11}^3=\omega_{33}^2=0, \ \ \ \ \omega_{21}^3=-\frac{1}{\sqrt{6}}, \ \ \ \ \omega_{12}^3=\frac{1}{\sqrt{6}},\ \ \ \ \ \omega_{31}^1=\frac{\sqrt{3} }{2}h_{22}^3+\frac{1}{\sqrt{6}}.
        \end{equation*}
    Computing the Codazzi equation \eqref{Codazzi} with $X=E_3$, $Y=E_2$, $Z=E_3$ and $X=E_1$, $Y=E_2$, $Z=E_2$ we obtain 
        \[
            \frac{4-3 h_{22}^3 \left(3 h_{22}^3+\sqrt{2}\right)}{3 \sqrt{3}}=0, \ \ \ \  8 \sqrt{3}-3 \sqrt{6}  h_{22}^3=0,
        \]
    which is a contradiction.

    Suppose now that $\theta_1=\theta_2$ modulo $\pi$. Using that $2\theta_1+\theta_2=0$ modulo $\pi$  we deduce $\theta_1=\theta_2=\pi/3$ or $2\pi/3$.
    By Lemma \ref{isometrieswithAB} we know that these cases are congruent via the isometry $\Psi_{0,1}$ given in \eqref{isoslsl}. Thus, we only consider the cases where $\theta_1=\theta_2=\pi/3$.

    By Lemma \ref{nablapcase2} we have $h_{22}^2=h_{12}^3=0$. 
    Moreover, by Proposition \ref{frameuniquecase2} we may assume that $h_{22}^1=0$ and that all the functions $h_{ij}^k$, $\omega_{ij}^k$ are constant. 
    Hence, from Equation \eqref{nablap} we obtain
        \[
            \omega_{11}^1=\omega_{11}^3=\omega_{22}^2=\omega_{33}^2=0.
        \] 
    We also get
        \[
            \omega_{21}^3=\sqrt{3} h_{22}^3-\frac{1}{\sqrt{6}}, \ \ \ \ \ \ \omega_{31}^1=\frac{\sqrt{3}}{2}h_{22}^3+\frac{1}{\sqrt{6}}.
        \]
        
    Thus computing the Codazzi equation with $X=E_3$, $Y=E_2$ and $Z=E_2$ yields $\omega_{33}^1=0$. 
    Moreover, we obtain
        \[
            -9 (h_{22}^3)^2+3 \sqrt{2}h_{22}^3+4=0.
        \]
    Hence $h_{22}^3=-\frac{\sqrt{2}}{3}$ or $h_{22}^3=\frac{2 \sqrt{2}}{3}$. 

    If $h_{22}^3=-\frac{ \sqrt{2}}{3}$ then the Codazzi equation with $X=E_1$, $Y=E_2$ and $Z=E_2$ implies $\omega_{12}^3=-\sqrt{\frac{3}{2}}$. 
    Also we obtain $\omega_{22}^3=0$. 

    If instead $h_{22}^3=\frac{2 \sqrt{2}}{3}$ we obtain $\omega_{12}^3=\sqrt{\frac{3}{2}}$. 
    A straightforward computation shows that both cases have constant sectional curvarture equal to $-\tfrac{3}{2}$.
\end{proof}
      
Now we exhibit two examples of extrinsically homogeneous Lagrangian submanifolds of type II in Lemma \ref{propAB}.

\begin{example}\label{type2asubmanifoldexample}
Let $\R\ltimes_{\varphi_0}\R^2$ be the Bianchi group of type V with group law given by
\[(t,u)\cdot(s,v)=(t+s, \varphi_0(s)u+v) \] 
where $\varphi_0\colon\R\to \mathrm{Aut}(\R^2)$ is given by
\[
    \varphi_0(s)=
    \begin{pmatrix}
    e^{-2s}&0\\
    0& e^{-2s}
\end{pmatrix}.
\]
Let $\hat{g}$ be the right invariant metric such that its components with respect to the frame of right invariant vector fields  $\left\{\pdv{}{t},e^{-2t}\pdv{}{u_1},e^{-2t}\pdv{}{u_2}\right\}$ are given by
\[\left(
\begin{array}{ccc}
 \frac{8}{3} & 0 & 0 \\
 0 & -\frac{3}{2} & 0 \\
 0 & 0 & \frac{3}{2} \\
\end{array}
\right).\]
In fact, this Lorentzian manifold is simply connected, geodesically complete and it has consant sectional curvature equal to $-\tfrac{3}{2}$. 
Hence, by the pseudo-Riemannian analogue of the Killing-Hopf theorem (see \cite{oneill}), it is isometric to $\tilde{H}^3_1(-\tfrac{3}{2})$, the universal cover of the anti-de Sitter space.

Now let $\mathfrak{h}$ be the Lie subalgebra of $\slf\oplus\slf\oplus\slf$ spanned by $\{e_1,e_2,e_3\}$ where 
\begin{equation*}
    \begin{split}
        e_1&=\Big(\ii,-\ii,-\ii\Big),\\
        e_2&=\left(\frac{9}{4}(\jj-\kk),\frac{1}{2}(\jj+\kk),0\right),\\
        e_3 &=\left(\frac{9}{4}(-\jj+\kk),0,\frac{1}{2}(\jj+\kk)\right),        
    \end{split}
\end{equation*}
where $\ii$, $\jj$ and $\kk$ are given in \eqref{ijksl2R}.
The Lie algebra $\mathfrak{h}$ is a Bianchi Lie algebra of type V with brackets
\[
[e_1,e_2]=-2e_2, \ \ \ [e_1,e_3]=-2e_3, \ \ \ [e_2,e_3]=0.
\]
Therefore $\R\ltimes_{\varphi_0}\R^2$ is the universal cover of $\exp(\mathfrak{h})\subset\Sl\times\Sl\times\Sl$. Moreover, the map 
\[
(w,u,v)\mapsto \exp( w e_1+ ue^{w } w \csch(w )e_2+v e^{w } w  \csch(w )e_3)
\]
defined at $w=0$ as $\exp(ue_2+v e_3)$, is a group isomorphism, thus $\exp \mathfrak{h}\cong\R\ltimes_{\varphi_0}\R^2 $. One can check that the immersion  $\iota:\R\ltimes_{\varphi_0}\R^2\to\Sl\times\Sl$ given by
\begin{equation}
    \iota(w,u,v)=\exp( w e_1+ ue^{w } w \csch(w )e_2+v e^{w } w  \csch(w )e_3)\cdot (\id_2,\id_2) \label{mapiota}
\end{equation}
is a Lagrangian immersion, whose image is extrinsically homogeneous. The frame given by
\[
E_1 =-e^{-2w}\iota_u-e^{-2w}\iota_v, \ \ \ \ E_2=\tfrac{1}{3}e^{-2w}\iota_u-\tfrac{1}{3}e^{-2w}\iota_v, \ \ \ \ E_3=\sqrt{\tfrac{3}{8}}\iota_w,
\]
is a $\Delta_2$-orthonormal frame with respect to which $A$ and $B$ take type II form in Lemma \ref{propAB}, with angle functions $\theta_1=\theta_2=\tfrac{\pi}{3}$.
\end{example}

\begin{example}\label{type2bsubexample}
    Let $\R\ltimes_{\varphi_1}\R^2$ be the Bianchi group of type III  with group law
\[
    (t,u)\cdot (s,v)=(t+s,\varphi_1(s)u+v)
\]
    where $\varphi_1\colon\R\to\mathrm{Aut}(\R^2)$ is given by
\[
    \varphi_1(s)=\begin{pmatrix}
        e^{2s} & 0\\
        0 & 1\\
    \end{pmatrix}.
\]
    Let $g_\lambda$ be the right invariant metric on  $\R\ltimes_{\varphi_1}\R^2$ such that its components with respect to the frame of right invariant vector fields $\left\{\pdv{}{t},e^{2t}\pdv{}{u_1},\pdv{}{u_2}\right\}$ are given by
\[
    \begin{pmatrix}
 \frac{2}{3} & 0 & 0 \\
 0 & 0 & 1 \\
 0 & 1 & \frac{2 (\lambda -1)}{3} \\
    \end{pmatrix},
\]
    where $\lambda$ is an arbitrary real number.
As in Example \ref{type2asubmanifoldexample}, this Lorentzian manifold is simply connected, geodesically complete and it has consant sectional curvature equal to $-\tfrac{3}{2}$. Therefore it is isometric to $\tilde{H}^3_1(-\tfrac{3}{2})$.

Let $\mathfrak{h}$ be the Lie subalgebra of $\slf\oplus\slf\oplus\slf$ spanned by $\{e_1,e_2,e_3\}$ where 
\begin{equation*}
    \begin{split}
e_1&=(\ii,0,0),\\
e_2&=(\tfrac{1}{2}(\jj+\kk),0,0),\\
e_3&=(0,-\tfrac{ \lambda +7}{6}\jj+\tfrac{11-\lambda }{6}\kk,-\tfrac{\lambda +9}{6}\jj+\tfrac{9-\lambda }{6}\kk).
    \end{split}
\end{equation*}
where $\ii$, $\jj$, $\kk$ are given in \eqref{ijksl2R}.
The Lie algebra $\mathfrak{h}$ is a Bianchi Lie algebra of type III with brackets
\[
[e_1,e_2]=2e_2, \ \ \ \ [e_1,e_3]=0, \ \ \ \ [e_2,e_3]=0.
\]
Therefore $\R\ltimes_{\varphi_1}\R^2$ is the universal cover of $\exp(\mathfrak{h})$. The map $\phi_\lambda\colon \R\ltimes_{\varphi_1}\R^2\to\exp(\mathfrak{h})$ given by
\[
\phi_\lambda(u,v,w)=\exp( w e_1+\frac{u w e^{-w}}{\sinh w} e_2+v e_3)
\]
with $\frac{w e^{-w}}{\sinh w}$ extended to $1$ when $w=0$, is a surjective homomorphism with 
\[
H_\lambda=\ker(\phi_\lambda)\cong\left\{\begin{array}{cc}
    \mathbb{Z}  & \text{ when $\lambda=\frac{2n^2}{n^2-m^2}$, $m>n>0$  integers,} \\
     \{0\}  & \text{otherwise}.
\end{array}\right.
\]
One can check that the map $f_\lambda\colon(\R\ltimes_{\varphi_1}\R^2)/H_\lambda\to\Sl\times\Sl$ given by
\[
f_\lambda(u,v,w)=\phi_\lambda(u,v,w)\cdot(\id_2,\id_2),
\]
is a Lagrangian immersion, whose image is extrinsically homogeneous. The frame $\{E_1,E_2,E_3\}$ given by
\[
E_1=e^{2w}(f_{\lambda})_u,\ \ \ \ \ E_2=\frac{1-\lambda}{3}e^{2w}(f_{\lambda})_u+  (f_{\lambda})_v, \ \ \ \ \ E_3=\sqrt{\frac{3}{2}}(f_{\lambda})_w
\]
is a $\Delta_2$-orthonormal frame with respect to which  $A$ and $B$ take type II form in Lemma \ref{propAB} with $\theta_1=\theta_2=\tfrac{\pi}{3}$.
\end{example}

\begin{remark}
    For any pair $\lambda_1, \lambda_2$ the subgroups $\phi_{\lambda_i}(\R\ltimes_{\varphi_1}\R^2)$ of $\Sl\times\Sl\times\Sl$ are non-conjugate. 
    That is, there does not exist an automorphism of $\Sl\times\Sl\times\Sl$ preserving the isotropy 
    subgroup $\Delta\Sl$
    that maps $\phi_{\lambda_1}(\R\ltimes_{\varphi_1}\R^2)$ into $\phi_{\lambda_2}(\R\ltimes_{\varphi_1}\R^2)$. 
    This can be easily seen since conjugations by elements of $\mathrm{SL}^\pm(2,\R)$ preserve the indefinite inner product of $\slf$ given in \eqref{prodsl2}.
   \end{remark}

\begin{proposition}\label{type2asubmanifold}
Let $f\colon M\to\Sl\times\Sl$ be an extrinsically homogeneous Lagrangian submanifold of the pseudo-nearly Kähler $\Sl\times\Sl$. Suppose that $A$ and $B$ take type II form in Lemma \ref{propAB} with respect to a $\Delta_2$-orthonormal frame $\{E_1,E_2,E_3\}$. Then $M$ is congruent to an open subset of either the image of the immersion in Example \ref{type2asubmanifoldexample}, or the image of the immersion in Example~\ref{type2bsubexample}.
\end{proposition}

\begin{proof}
    Because of Proposition \ref{propositiontype2twotypes} we may assume that $\theta_1=\theta_2=\tfrac{2}{3}\pi$. Moreover, we divide in two cases. 

Suppose that $\omega_{ij}^k$ and $h_{ij}^k$ satisfy (\ref{typeIIfirst}) in Proposition \ref{propositiontype2twotypes}.
We take the frame $\{\rho E_1,\rho E_2,E_3\}$ where $\rho$ is a non-vanishing smooth function and solution of 
\begin{equation}
    E_1(\rho)=E_2(\rho)=0,\ \ \ \  E_3(\rho)=\sqrt{\tfrac{3}{2}}\rho.\label{rhoeqtype1}
\end{equation}
 It is easy to check that $\rho$ indeed exists and that 
\[
[\rho E_1,\rho E_2]=[\rho E_1, E_3]=[\rho E_2, E_3]=0.
\]
Hence, there exist local coordinates ${u,v,w}$ such that $\rho E_1=f_u$, $\rho E_2=f_v$ and $E_3=f_w$ and hence $\rho(w)=e^{\sqrt{\tfrac{3}{2}}w}$.
It follows from the relation between $Q$ and $P$ given in \eqref{prodQ} that
\begin{equation*}
    Qf_u=-f_u, \ \ \ Qf_v=-\frac{2}{3}f_u-f_v+\frac{2}{\sqrt{3}}Jf_u, \ \ \ Qf_w=-f_w. 
\end{equation*}
Writing $f=(p,q)$ yields 
\begin{equation}
   q_u=q_w=0, \ \ \ \  q_v=\tfrac{2}{3}qp^{-1}p_u.\label{qv}
\end{equation}
Using the relation between the nearly Kähler connection $\tilde{\nabla}$ and the product connection $\nabla^E$ given in \eqref{relprodkal} we obtain
\begin{equation}
    \begin{aligned}
        \nabla^E_{f_u}f_u&=0,      &&\nabla^E_{f_u}f_v=-\sqrt{\tfrac{3}{2}} \rho ^2f_w,  &&\nabla^E_{f_u}f_w\sqrt{\tfrac{3}{2}}f_u,\\
        \nabla^E_{f_v}f_u&=-\sqrt{\tfrac{3}{2}} \rho ^2f_w, \ \ \ \  &&  \nabla^E_{f_v}f_v= -\sqrt{\tfrac{2}{3}} \rho ^2f_w, \ \ \ \ &&\nabla^E_{f_v}f_w=\tfrac{1}{\sqrt{6}}f_u+\sqrt{\tfrac{3}{2}} f_v-\tfrac{1 }{\sqrt{2}}Jf_u, \\
        \nabla^E_{f_w}f_u&=\sqrt{\tfrac{3}{2}}f_u, && \nabla^E_{f_w}f_w=0, &&\nabla^E_{f_w}f_v=\tfrac{1 }{\sqrt{6}}f_u+\sqrt{\tfrac{3}{2}} f_v-\tfrac{1 }{\sqrt{2}}Jf_u. \label{prdoconex1}
    \end{aligned}
\end{equation}

Now we compute 
\begin{equation}
    \begin{aligned}
        \li f_u,f_u\ri&=0, \ \ \ \ \ \ \ &&\li f_v,f_v\ri=\rho^2,\ \ \ \ \ \ \  &&\li f_w,f_w\ri=\tfrac{3}{2}, \\
        \li f_u,f_v\ri&=\tfrac{3}{2}\rho^2, &&\li f_u,f_w\ri=0, &&\li f_v,f_w\ri=0,
    \end{aligned}\label{metricsf}
\end{equation}
and
\begin{equation}
    \begin{aligned}
        \li f_u,Qf_u\ri&=0, \ \ \ \ &&\li f_u,Qf_w\ri=0, \ \ \ \ &&\li f_v,Qf_w\ri=0, \\
        \li f_u,Qf_v\ri&=-\tfrac{3}{2}\rho^2, &&\li f_v,Qf_v\ri=-\rho^2, &&\li f_w,Qf_w\ri=-\tfrac{3}{2},
    \end{aligned}\label{metricsq}
\end{equation}
where $\li,\ri$ is the product metric associated to the metric on $\Sl$ given in \eqref{prodsl2}.
In particular, we have
\begin{equation}
    \begin{aligned}
        \li p_u,p_u\ri&=0,  &&\li p_v,p_v\ri=\rho^2, && \li p_w,p_w\ri=\tfrac{3}{2}\\
        \li p_u,p_v\ri&=\tfrac{3}{2}\rho^2,  &&\li p_u,p_w\ri=0, && \li p_v,p_w\ri=0.\\
    \end{aligned}\label{metricspex1case2}
\end{equation}
Here, $\li,\ri$ is the metric on $\Sl$ given in \eqref{prodsl2}.

To compute the second derivatives of $f$, we use the expression for the Euclidean connection of $\R^8_4$ in \eqref{connectionr8}. 
Plugging \eqref{prdoconex1}, \eqref{metricsf} and \eqref{metricsq} into \eqref{connectionr8} we obtain
\begin{equation*}
    \begin{aligned}
        f_{uu}&=0, &&f_{uv}=-\sqrt{\tfrac{3}{2}} \rho ^2f_w+\tfrac{3}{4}\rho^2f-\tfrac{3}{4}\rho^2Qf,\\
        f_{uw}&=\sqrt{\tfrac{3}{2}}f_u,  && f_{vv}=-\sqrt{\tfrac{2}{3}}\rho^2 f_w +\tfrac{1}{2}\rho^2f-\tfrac{1}{2}\rho^2Qf,\\
        f_{vw}&=\tfrac{1 }{\sqrt{6}}f_u+\sqrt{\tfrac{3}{2}} f_v-\tfrac{1 }{\sqrt{2}}Jf_u,\ \ \ \ \ \  &&f_{ww}=\tfrac{3}{4}f-\tfrac{3}{4}Qf.
    \end{aligned}
\end{equation*}
Looking at each component of $f$ we obtain differential equations for $p$ and $q$:
\begin{equation}
    \begin{aligned}
     p_{uu}&=0, &&p_{uv}=-\sqrt{\tfrac{3}{2}}\rho^2p_w+\tfrac{3}{2}\rho^2 p, \ \ \ \ \ &&p_{uw}=\sqrt{\tfrac{3}{2}}p_u, &&\\
     p_{vw}&=\sqrt{\tfrac{3}{2}}p_v, \ \ \ \ \  &&  p_{vv}=-\sqrt{\tfrac{2}{3}}\rho^2p_w+\rho^2p,&&p_{ww}=\tfrac{3}{2}p,  &&q_{vv}=0.\label{difeqpq}
    \end{aligned}
\end{equation}
The other derivatives of $q$ are zero because of \eqref{qv}.
 
Applying an isometry of the type $(p,q)\mapsto(ap,bq)$, we may assume initial conditions $(p(0),q(0))=(\id_2,\id_2)$. Then from \eqref{metricspex1case2} and \eqref{qv} it follows that there exists $c\in\mathrm{SL}^\pm(2,\R)$ such that
\begin{equation*}
    \begin{aligned}
        p_u(0)&=c\left(
\begin{array}{cc}
 0 & 1 \\
 0 & 0 \\
\end{array}
\right)c^{-1}, && p_v(0)=c\left(
\begin{array}{cc}
 0 & \frac{1}{3} \\
 3 & 0 \\
\end{array}
\right)c^{-1}, \\
p_w(0)&=\sqrt{\frac{3}{2}}c\left(
\begin{array}{cc}
 1 & 0 \\
 0 & -1 \\
\end{array}
\right)c^{-1}, && q_v(0)=c\left(
\begin{array}{cc}
 0 & \frac{2}{3} \\
 0 & 0 \\
\end{array}
\right)c^{-1}.
    \end{aligned}
\end{equation*}
 Applying the isometry  $(p,q)\mapsto(cpc^{-1},cqc^{-1})$ we obtain that any solution of \eqref{difeqpq} is congruent to an open subset of the immersion $f=(p,q)$ where
\begin{equation*}
p(u,v,w)=\left(
\begin{array}{cc}
 e^{\sqrt{\frac{3}{2}} w} & e^{\sqrt{\frac{3}{2}} w} \left(u+\frac{v}{3}\right) \\
 3 e^{\sqrt{\frac{3}{2}} w} v & e^{\sqrt{\frac{3}{2}} w} \left(v^2+3 u v\right)+e^{-\sqrt{\frac{3}{2}} w} \\
\end{array}
\right), \ \ \ \ 
q(v)=\left(
\begin{array}{cc}
 1 & \frac{2 v}{3} \\
 0 & 1 \\
\end{array}
\right).
\end{equation*}
Finally, taking the change of coordinates $w\to2 \sqrt{\frac{2}{3}}w$, $u\to-\frac{1}{2} (u+v)$ and $v\to\frac{3}{2} (u-v) $ we get the immersion in \eqref{mapiota}.

Now suppose that $\omega_{ij}^k$ and $h_{ij}^k$ satisfy (\ref{typeIIsecond}) in Proposition \ref{propositiontype2twotypes}.
In this case, $\omega_{22}^3$ is constant. 
We define the constant $\lambda$ as $\omega_{22}^3=\sqrt{\frac{2}{3}} (1-\lambda )$.   

Take the frame $\left\{\rho E_1,-\tfrac{1}{\sqrt{6}}\omega_{22}^3 E_1+E_2,E_3\right\}$, where $\rho$ is a non-vanishing smooth function and solution of 
\[E_1(\rho)=E_2(\rho)=0, \ \ \ \ E_3(\rho)=-\sqrt{6}E_3.\]
Using Proposition \ref{propositiontype2twotypes} we can easily check that this is a coordinate frame. 
We call this frame $\{f_u,f_v,f_w\}$.
First we notice that $\rho=e^{-\sqrt{6}w}$. We obtain from Equation \eqref{prodQ} that
\begin{equation*}
    \begin{split}
        Qf_u=-f_u, \ \ \ Qf_v=-\tfrac{2}{3\rho}f_u-f_v+\tfrac{2}{\sqrt{3}\rho}Jf_u, \ \ \ Qf_w=-f_w,
    \end{split}
\end{equation*}
We deduce 
\begin{equation}
q_u=q_w=0, \ \ \ \ \ q_v=\tfrac{2}{3\rho}qp^{-1}p_u.   \label{qvlambda}
\end{equation}
Equation \eqref{relprodkal} gives us the following expressions
\begin{equation*}
    \begin{aligned}
    \nabla^E_{f_u}f_u&=0, && \nabla^E_{f_u}f_v=\sqrt{\frac{3}{2}} \rho f_w, \ \ \ \ && \nabla^E_{f_v}f_v=-\sqrt{\frac{2}{3}}f_w+\sqrt{2}Jf_w, \\
    \nabla^E_{f_u}f_w&=-\sqrt{\frac{3}{2}}f_u, \ \ \  \  && \nabla^E_{f_w}f_w=0, &&
    \nabla^E_{f_v}f_w=\tfrac{2 \lambda-1}{\sqrt{6}\rho}f_u-\sqrt{\tfrac{3}{2}}f_v+\tfrac{1}{\sqrt{2}\rho}Jf_u. 
    \end{aligned}
\end{equation*}
From Equation \eqref{prodmetric} it follows
\begin{equation*}
    \begin{aligned}
        \li f_u,f_u\ri&=0, \ \ \ &&\li f_u,f_v\ri=\tfrac{3\rho}{2}, \ \ \ \ &&\li f_u,f_w\ri=0,\\
        \li f_v,f_v\ri&=\lambda, \ \ \ &&\li f_v,f_w\ri=0, \ \ \ \ &&\li f_w,f_w\ri=\tfrac{3}{2}.
    \end{aligned}
\end{equation*}
and 
\begin{equation*}
    \begin{aligned}
        \li f_u,Qf_u\ri&=0, \ \ \ &&\li f_u,Qf_v\ri=-\tfrac{3\rho}{2}, \ \ \ \ &&\li f_u,Qf_w\ri=0,\\
        \li f_v,Qf_v\ri&=-\lambda, \ \ \ &&\li f_v,Qf_w\ri=0, \ \ \ \ &&\li f_w,Qf_w\ri=-\tfrac{3}{2}.
    \end{aligned}
\end{equation*}
In particular, we have 
\begin{equation}
    \begin{aligned}
        \li p_u,p_u\ri&=0, &&\li p_v,p_v\ri=\lambda, && \li p_w,p_w\ri=\tfrac{3}{2},\\
        \li p_u,p_v\ri&=\tfrac{3}{2}\rho, &&\li p_u,p_w\ri=0, && \li p_v,p_w\ri=0.\\ 
    \end{aligned}\label{metricspex2case2}
\end{equation}
We may use Equation \eqref{connectionr8} to compute
\begin{equation*}
    \begin{aligned}
        f_{uu}&=0, &&f_{vv}=-\sqrt{\frac{2}{3}}f_w+\sqrt{2}Jf_w+\frac{1}{2} \lambda f-\frac{1}{2}\lambda Qf, &&f_{ww}=\tfrac{3}{4}f-\tfrac{3}{4}Qf,\\
        f_{uw}& =-\sqrt{\frac{3}{2}}f_u, &&     f_{vw}=\tfrac{2 \lambda-1}{\sqrt{6}\rho}f_u-\sqrt{\tfrac{3}{2}}f_v+\tfrac{1}{\sqrt{2}\rho}Jf_u, && f_{uv}=\sqrt{\frac{3}{2}} \rho f_w+\frac{3 \rho}{4}f-\frac{3 \rho}{4}Qf.
    \end{aligned}
\end{equation*}
Hence we obtain
\begin{equation}
    \begin{aligned}
    p_{uu}&=0, && p_{uv}=\sqrt{\tfrac{3}{2}}\rho p_w+\tfrac{3}{2}\rho p, && p_{uw}=-\sqrt{\tfrac{3}{2}}p_u, \ \ \ &&\\
    p_{vv}&=\lambda p, \ \ \ \  && p_{vw}=\tfrac{1}{\rho}\sqrt{\tfrac{2}{3}} \lambda p_u-\sqrt{\tfrac{3}{2}}p_v, \ \ \ \ &&  p_{ww}=\tfrac{3}{2}p, && q_{vv}=2\sqrt{\tfrac{2}{3}}qp^{-1}p_w.\\
    \end{aligned} \label{difeqpqex2}
\end{equation}
Applying an isometry of the type $(p,q)\mapsto (ap,bq)$ we may assume that $p(0)=\id_2$ and $q(0)=\id_2$. From \eqref{metricspex2case2} and \eqref{qvlambda} it follows that there exists a matrix $c\in\mathrm{SL}^\pm(2\R)$ such that 
\begin{equation*}
    \begin{aligned}
       p_u(0)&= c\left(
\begin{array}{cc}
 0 & 1 \\
 0 & 0 \\
\end{array}
\right) c^{-1}, &&  p_v(0)=c\left(
\begin{array}{cc}
 0 & \frac{\lambda }{3} \\
 3 & 0 \\
\end{array}
\right)c^{-1}, \\ 
p_w(0)&=\sqrt{\frac{3}{2}}c\left(
\begin{array}{cc}
 1 & 0 \\
 0 & -1 \\
\end{array}
\right)c^{-1}, &&  q_v(0)=c\left(
\begin{array}{cc}
 0 & \frac{2}{3} \\
 0 & 0 \\
\end{array}
\right)c^{-1}.
    \end{aligned}
\end{equation*}
Applying the isometry $(p,q)\mapsto(cp c^{-1},cqc^{-1})$ of $\Sl\times\Sl$ we obtain that any solution of \eqref{difeqpqex2} is congruent to the solution with $c=\id_2$. 
After the change of coordinates $w\to \sqrt{\frac{3}{2}} w$, $u\to u e^{-\sqrt{\frac{3}{2}} w}$, we obtain that such solution is the immersion $f_\lambda$ in Example \ref{type2bsubexample}.
\end{proof}

\subsection{Extrinsically homogeneous Lagrangian submanifolds of type III}

\begin{example}\label{type3submanifold}
Let $\R\ltimes_{\varphi_2}\R^2$ be the Bianchi group of type VI with the group law
\[
    (t,u)\cdot(s,v)=(t+s, \varphi_2(s)u+v) 
\] 
where $\varphi_2\colon \R\to\mathrm{Aut}(\R^2)$ is given by
\[\varphi_2(s)=\begin{pmatrix} 
    e^{-2s}&0\\
    0& e^s\\
\end{pmatrix}.\]
Let $\tilde{g}$ be the right invariant metric such that its components with respect to the frame of right invariant vector fields  $\left\{\pdv{}{t},e^{-2t}\pdv{}{u_1},e^t\pdv{}{u_2}\right\}$ are given by
\[\left(
\begin{array}{ccc}
 0 & 2 & \frac{7}{3} \sqrt{\frac{2}{3}} \\
 2 & 0 & 8 \sqrt{\frac{2}{3}} \\
 \frac{7}{3} \sqrt{\frac{2}{3}} & 8 \sqrt{\frac{2}{3}} & \frac{128}{9} \\
\end{array}
\right).\]

Now let $\mathfrak{h}$ be the Lie subalgebra of $\slf\oplus\slf\oplus\slf$ spanned by $\{e_1,e_2,e_3\}$ where 
\begin{equation*}
    \begin{split}
           e_1&=\bigg(\tfrac{1}{18} \left(27+2 \sqrt{6}\right)\ii-\Big(2 \sqrt{\tfrac{2}{3}}+\tfrac{3}{4}\Big)\jj+\Big(\tfrac{8}3 \sqrt{\tfrac{2}{3}}+\tfrac{3}{4}\Big)\kk,\\
&\quad-\Big(1+\tfrac{17}{12 \sqrt{6}}\Big)\ii+\tfrac{1}{96} \left(48-17 \sqrt{6}\right)\jj-\Big(\tfrac{1}{2}+\tfrac{85}{48 \sqrt{6}}\Big)\kk,\\
&\quad-\tfrac{1}{2}\ii+\tfrac{1}{4} \left(1-3 \sqrt{6}\right)\jj+\tfrac{1}{4} \left(3 \sqrt{6}-1\right)\kk\bigg),\\
        e_2&=\left(0,\sqrt{\tfrac{2}{3}}\ii+\tfrac{1}{2}\sqrt{\tfrac{3}{2}}\jj+\tfrac{5}{2 \sqrt{6}}\kk,0\right),\\
        e_3&=\left(\tfrac{8}{9} \left(2+3 \sqrt{6}\right)\ii-\tfrac{2}{3} \left(7+2 \sqrt{6}\right)\jj+\tfrac{2}{9} \left(37+6 \sqrt{6}\right)\kk,0,-6\jj+6\kk\right).
    \end{split}
\end{equation*}
The Lie algebra $\mathfrak{h}$ is a Bianchi Lie algebra of type VI with brackets
\[
[e_1,e_2]=-2e_2, \ \ \ [e_1,e_3]=e_3, \ \ \ [e_2,e_3]=0.
\]
Therefore $\R\ltimes_{\varphi_2}\R^2$ is the universal cover of $\exp(\mathfrak{h})$. Moreover, the map defined as
\[
(v,u,w)\mapsto \exp( v e_1+ \frac{2 u e^{2 v} v}{e^{2 v}-1}e_2+ \frac{v w}{e^v-1}e_3)
\]
and as $ \exp( 2u e_2+ w e_3)$ when $v=0$,
is a group isomorphism. Thus $\exp(\mathfrak{h})\cong \R\ltimes_{\varphi_2}\R^2 $. 
One can check that the immersion  $\jmath:\R\ltimes_{\varphi_2}\R^2\to\Sl\times\Sl$ given by
\[
\jmath(u,v,w)=\exp( v e_1+ \frac{2 u e^{2 v} v}{e^{2 v}-1}e_2+ \frac{v w}{e^v-1}e_3)\cdot (\id_2,\id_2)
\]
is a Lagrangian immersion, whose image is extrinsically homogeneous. 
The frame given by
\[
E_1 =-\sqrt{\frac{3}{2}} e^{-2 v}\jmath_u, \ \ \ \ E_2=-\frac{1}{\sqrt{6}}\jmath_v, \ \ \ \ E_3=\frac{7 e^{-2 v}}{4 \sqrt{6}} \jmath_u +\sqrt{6}\jmath_v-\frac{3}{4} e^v\jmath_w.
\]
is a $\Delta_2$-orthonormal frame with respect to which $A$ and $B$ take type III form from Lemma \ref{propAB}.
\end{example}
\begin{proposition}\label{type3prop}
    Let $M$ be an extrinsically homogeneous Lagrangian submanifold of $\Sl\times\Sl$. 
    Suppose that $A$ and $B$ take type III form from Lemma \ref{propAB} with respect to a $\Delta_2$-orthonormal frame $\{E_1,E_2,E_3\}$.
     Then $M$ is congruent to an open subset of the submanifold given in Example~\ref{type3submanifold}.
\end{proposition}
\begin{proof}
By Lemma \ref{isometrieswithAB} we may apply the isometry $\Psi_{0,1}$ in \eqref{isoslsl} and assume that the sign of $B$ is $-1$.
Proposition \ref{constanttype3} implies that the components of the second fundamental form and of the connection associated to $E_1$, $E_2$ and $E_3$ are all constant.
From the Codazzi equation~\eqref{Codazzi} with $X=E_1$, $Y=E_2$ and $Z=E_2$  we have $ h_{22}^2=\frac{2\sqrt{2}}{3}$. 
Then, computing the Codazzi equation with $X=E_3$, $Y=E_2$ and $Z=E_2$ we obtain
\[
h_{22}^1=-\frac{13}{18\sqrt{2}},\ \ \ \ h_{22}^3=\frac{5\sqrt{2}}{9}.
\]
We define the frame $\left\{\rho E_1,E_2,\rho^{-1/2}(\frac{7 }{12 \sqrt{6}}E_1+\sqrt{6} E_2+\frac{1}{\sqrt{6}}E_3)\right\}$ where $\rho$ is a non-vanishing smooth function that satisfies
\[
E_1(\rho)=0, \ \ E_2(\rho)=\sqrt{\tfrac{2}{3}}\rho, \ \ E_3(\rho)=-2\sqrt{6}\rho.
\]
From \eqref{relationsomegatype3} we can check that $\rho$ indeed exists and that the defined frame is a coordinate frame.
We write said frame as $\{f_u,f_v,f_w\}$.
Hence $\rho=e^{\sqrt{\frac{2}{3}} v}$.
Equation \eqref{prodQ} yields
\begin{equation*}
    \begin{split}
        Qf_u&=f_u,\\
        Qf_v&=-\frac{23}{18 \rho }f_u-3f_v+2 \sqrt{\frac{2\rho}{3}} f_w-\frac{7}{6 \sqrt{3} \rho }Jf_u-4 \sqrt{3} Jf_v+2 \sqrt{2\rho }Jf_w,\\
        Qf_w&=-\frac{7\rho ^{-3/2}}{\sqrt{6} }f_u-4\sqrt{\frac{6}{\rho }}f_v+5f_w-\frac{5\rho ^{-3/2}}{3 \sqrt{2} }Jf_u-12\sqrt{\frac{2}{\rho }}Jf_v+4\sqrt{3}Jf_w.
    \end{split}
\end{equation*}
If we denote $f=(p,q)$ then the first equation implies that $p_u=0$. From the second and third equations we obtain
\begin{equation}
p_v=pq^{-1}\left(-\frac{5q_u}{6 \rho}-4 q_v+2 \sqrt{\frac{2}{3}} \sqrt{\rho}q_w\right), \ \ \ p_w=pq^{-1}\left(4 q_w-\frac{13q_u+72 \rho q_v}{3 \sqrt{6} \rho ^{3/2}}\right). \label{pinitialconditionstype3}    
\end{equation}
The relation between the connection $\nabla^E$ associated to the product metric and the nearly Kähler connection $\tilde{\nabla}$ given in Equation \eqref{relprodkal} yields
\begin{equation*}
    \begin{split}
        \nabla^E_{f_u}f_u&=0, \\
        \nabla^E_{f_u}f_v&=-\tfrac{11}{4 \sqrt{6}}f_u-3 \sqrt{6} \rho f_v+3 \rho ^{3/2}f_w+\tfrac{1}{\sqrt{2}}Jf_u,\\
       \nabla^E_{f_u}f_w &=-\tfrac{13}{4 \sqrt{\rho }}f_u-18 \sqrt{\rho }f_v+3 \sqrt{6} \rho f_w+\sqrt{\tfrac{3}{\rho }} Jf_u,\\
        \nabla^E_{f_v}f_v&=\tfrac{49}{48 \sqrt{6} \rho }f_u+\tfrac{29}{2 \sqrt{6}}f_v-\tfrac{7 \sqrt{\rho }}{4} f_w-\tfrac{17}{9 \sqrt{2} \rho }Jf_u-5 \sqrt{2}Jf_v+2 \sqrt{3\rho }Jf_w,\\
        \nabla^E_{f_v}f_w&=\tfrac{91}{96 \rho ^{3/2}}f_u+\tfrac{53}{4 \sqrt{\rho }}f_v-\tfrac{43}{4 \sqrt{6}}f_w-\tfrac{25}{6 \sqrt{3} \rho ^{3/2}}Jf_u-7 \sqrt{\tfrac{3}{\rho }}Jf_v+\tfrac{9}{\sqrt{2}}Jf_w,\\
        \nabla^E_{f_w}f_w&=\tfrac{19}{3 \sqrt{6} \rho ^2} f_u+\tfrac{12 \sqrt{6}}{\rho }f_v-\tfrac{10}{\sqrt{\rho }}f_w-\tfrac{17}{3 \sqrt{2} \rho ^2}Jf_u-\tfrac{12 \sqrt{2}}{\rho }Jf_v+6 \sqrt{\tfrac{3}{\rho }}Jf_w.
    \end{split}
\end{equation*}
From the relation between the product metric and the nearly Kähler metric in Equation \eqref{prodmetric} it follows
\begin{equation}
    \begin{aligned}
        \li f_u,f_u\ri&=0, && \li f_v,f_v\ri=0, && \li f_w,f_w\ri=\tfrac{4}{\rho }, \\
        \li f_u,f_v\ri&=\tfrac{3 }{2}\rho,\ \ \  && \li f_u,f_w\ri=3 \sqrt{\tfrac{3\rho}{2} }, \ \ \  && \li f_v,f_w\ri=\tfrac{5 }{8}\sqrt{\tfrac{3}{2 \rho }}.
    \end{aligned}\label{metricsfIII}
\end{equation}
and
\begin{equation}
    \begin{aligned}
        \li f_u,Qf_u\ri&=0, && \li f_v,Qf_v\ri=-\tfrac{4}{3}, && \li f_w,Qf_w\ri=-\tfrac{4}{\rho }, \\
       \li f_u,Qf_v\ri&=\tfrac{3}{2}\rho,\ \ \  && \li f_u,Qf_w\ri=3 \sqrt{\tfrac{3\rho}{2} }, \ \ \ && \li f_v,Qf_w\ri=-\tfrac{49}{8 \sqrt{6\rho }} .
    \end{aligned}\label{metricsqIII}
\end{equation}
In particular we have
\begin{equation}
    \begin{aligned}
        \li q_u,q_u\ri&=0, && \li q_v,q_v\ri=-\tfrac{2}{3}, && \li q_w,q_w\ri=0,\\
        \li q_u,q_v\ri&=\tfrac{3}{2}\rho, \ \ \ \ && \li q_u,q_w\ri=3 \sqrt{\tfrac{3\rho}{2} },\ \ \ \  && \li q_v,q_w\ri=-\tfrac{17}{8 \sqrt{6\rho }}.\\
    \end{aligned}\label{metricsqq}
\end{equation}
Now we compute 
\begin{equation*}
    \begin{split}
        f_{uu}&=0, \\
        f_{uv}&=-\tfrac{11}{4 \sqrt{6}}f_u-3 \sqrt{6} \rho f_v+3 \rho ^{3/2}f_w+\tfrac{1}{\sqrt{2}}Jf_u+\tfrac{3 \rho }{4}f+\tfrac{3 \rho }{4}Qf,\\
        f_{uw}&=-\tfrac{13}{4 \sqrt{\rho }}f_u-18 \sqrt{\rho }f_v+3 \sqrt{6} \rho f_w+\sqrt{\tfrac{3}{\rho }} Jf_u+\tfrac{3}{2} \sqrt{\tfrac{3}{2}} \sqrt{\rho} f+\tfrac{3}{2} \sqrt{\tfrac{3\rho}{2}} Qf,\\
        f_{vv}&=\tfrac{49}{48 \sqrt{6} \rho }f_u+\tfrac{29}{2 \sqrt{6}}f_v-\tfrac{7 \sqrt{\rho }}{4} f_w-\tfrac{17}{9 \sqrt{2} \rho }Jf_u-5 \sqrt{2}Jf_v+2 \sqrt{3\rho }Jf_w-\tfrac{2}{3}Qf,\\
        f_{vw}&=\tfrac{91}{96 \rho ^{3/2}}f_u+\tfrac{53}{4 \sqrt{\rho }}f_v-\tfrac{43}{4 \sqrt{6}}f_w-\tfrac{25}{6 \sqrt{3} \rho ^{3/2}}Jf_u-7 \sqrt{\tfrac{3}{\rho }}Jf_v+\tfrac{9}{\sqrt{2}}Jf_w+\tfrac{5}{16} \sqrt{\tfrac{3}{2\rho}} f-\tfrac{49}{16 \sqrt{6\rho }} Qf,\\
        f_{ww}&=\tfrac{19}{3 \sqrt{6} \rho ^2} f_u+\tfrac{12 \sqrt{6}}{\rho }f_v-\tfrac{10}{\sqrt{\rho }}f_w-\tfrac{17}{3 \sqrt{2} \rho ^2}Jf_u-\tfrac{12 \sqrt{2}}{\rho }Jf_v+6 \sqrt{\tfrac{3}{\rho }}Jf_w+\tfrac{2}{\rho }f-\tfrac{2}{\rho }Qf.
    \end{split}
\end{equation*}
Hence $q$ satisfies
\begin{equation}
    \begin{split}
        q_{uu}&=0,\\
        q_{ww}&=0,\\
        q_{uv}&=\frac{1}{8} \left(12 \rho  \left(q-2 \sqrt{6}q_v+2 \sqrt{\rho } q_w\right)-5 \sqrt{6} q_u   \right),\\
        q_{uw}&=\frac{1}{4 \sqrt{\rho }}\left(6 \rho  \left(\sqrt{6} q-12 q_v+2 \sqrt{6\rho } q_w\right)-17 q_u\right),\\
        q_{vv}&=\frac{1}{864 \rho }\left(323 \sqrt{6} q_u-72 \rho  \left(8 q-17 \sqrt{6} q_v+13 \sqrt{\rho } q_w\right)\right),\\
        q_{vw}&=\frac{1}{288 \rho ^{3/2}}\left(289 q_u-6 \rho  \left(17 \sqrt{6} q-204 q_v+30 \sqrt{6\rho }q_w\right)\right),\\
    \end{split}\label{equationsq}
\end{equation}
and $p$ satisfies
\begin{equation}
    \begin{split}
        p_{vv}&=\frac{2 }{3}p-\frac{1}{18 \sqrt{6} \rho }\left(19 pq^{-1}q_u+72 \rho  pq^{-1}q_v\right),\\
        p_{vw}&=\frac{1}{6 \rho ^{3/2}}\left(4 \rho  \left(\sqrt{6} p-6pq^{-1} q_v\right)-7pq^{-1} q_u\right),\\
        p_{ww}&=\frac{1}{9 \rho ^2}\left(36 \rho  \left(p-2 \sqrt{6} pq^{-1}q_v+\sqrt{\rho } pq^{-1}q_w\right)-17 \sqrt{6} pq^{-1}q_u\right).
    \end{split} \label{equationsp}
\end{equation}
After applying an isometry of the type $(p,q)\mapsto(ap,bq)$, we may assume that $p(0)=\id_2$ and $q(0)=\id_2$. 
From \eqref{pinitialconditionstype3} and \eqref{metricsqq} it follows that there exists a matrix $c\in \mathrm{SL}^\pm(2,\R)$ such that 
\begin{equation*}
    \begin{split}
        q_u(0)&=c\left(
\begin{array}{cc}
 1 & 2 \\
 -\frac{1}{2} & -1 \\
\end{array}
\right)c^{-1}, \\
 q_w(0)&=c\left(
\begin{array}{cc}
 0 & 0 \\
 3 \sqrt{\frac{3}{2}} & 0 \\
\end{array}
\right)c^{-1}, \\
\end{split}
\end{equation*}
\begin{equation*}
    \begin{split}
q_v(0)&=c\left(
\begin{array}{cc}
    \frac{1}{72} \left(-17+ 6 \sqrt{6}\right) & -\frac{17}{36} \\
    \frac{1}{144} \left(233-12  \sqrt{6}\right) & \frac{1}{72} \left(17-6 \sqrt{6}\right) \\
   \end{array}
\right)c^{-1},\\
 p_v(0)&=c\left(
\begin{array}{cc}
    \frac{1}{9}-\sqrt{\frac{2}{3}}%
    & \frac{2}{9} \\
     \sqrt{\frac{2}{3}}-\frac{1}{18} &\sqrt{\frac{2}{3}}-\frac{1}{9} \\
   \end{array}
\right)c^{-1}, \\ p_w(0)&=c\left(
\begin{array}{cc}
  -2+\frac{2}{9}\sqrt{6} & \frac{4}{3} \sqrt{\frac{2}{3}} \\
 2 -\frac{1}{3}\sqrt{\frac{2}{3}} & 2 -\frac{2}{9}\sqrt{6} \\
\end{array}
\right)c^{-1}.
    \end{split}
\end{equation*}
By applying the isometry $(p,q)\mapsto (c p c^{-1},c q c^{-1})$ we obtain that any solution of the system of differential equations given in \eqref{equationsq} and \eqref{equationsp} is congruent to the solution with $c=\id_2$. This solution is the map given in Example \ref{type3submanifold} after the change of coordinates $v\to\frac{v}{\sqrt{6}}$, $u\to\sqrt{\frac{3}{2}} u$, $w\to\frac{1}{4} \sqrt{\frac{3}{2}}w$.
   
\end{proof}

\subsection{Extrinsically homogeneous Lagrangian submanifolds of type IV}
\begin{proposition}\label{propnolagrtypeiv}
There are no extrinsically homogeneous Lagrangian submanifolds of the pseudo-nearly Kähler $\Sl\times\Sl$ of type IV in Lemma \ref{propAB}.    
\end{proposition}
\begin{proof}
    Let $M$ be an extrinsically homogeneous Lagrangian submanifold of the pseudo-nearly Kähler $\Sl\times\Sl$. 
    Suppose that $A$ and $B$ take type IV form with respect to a $\Delta_3$-orthonormal frame $\{E_1,E_2,E_3\}$.
    Then by Proposition \ref{frameuniquecase4} the functions $\psi$, $\theta_1$, $\theta_2$, $h_{ij}^k$ and $\omega_{ij}^k$ are constant. 
    Thus by Lemma \ref{nablapfourthcase} the functions $h_{ij}^k$ are all zero except for $h_{11}^3$. Recall that we can write $\theta_2=-2\theta_1 +k \pi$, with $k=0,1$. 
    Computing the Codazzi equation~\eqref{Codazzi} with $X=E_3$, $Y=E_1$, $Z=E_1$ and $X=E_1$, $Y=E_2$, $Z=E_2$ yields
    \begin{equation}
        (h_{11}^3)^2=\frac{1}{3} (\cosh \psi -\cos 6 \theta_1) \cosh \psi ,    \label{h113cuadrado}
    \end{equation}
    and
   \begin{equation}
        \sqrt{\frac{2}{3}} h_{11}^3=\frac{4 (-1)^{k } \sinh (\psi ) \left(-\cos 6 \theta_1 \cosh \psi +3 (h_{11}^3)^2+\cosh ^2(\psi )\right)}{3 (\cos 6 \theta_1-\cosh \psi)}\label{h1132}.
   \end{equation}
    Plugging \eqref{h113cuadrado} into \eqref{h1132} gives
    \[
    h_{11}^3=2 \sqrt{\frac{2}{3}} (-1)^{k +1} \sinh 2 \psi .
    \]
    Comparing both values of $(h_{11}^3)^2$ we derive $\cos 6\theta_1=9 \cosh \psi -8 \cosh3 \psi $. Then writing the Codazzi equation~\eqref{Codazzi} with $X=E_1$, $Y=E_3$ and $Z=E_2$ we obtain $\sinh 2\psi=0$. This is a contradiction since by Lemma \ref{propAB}, $\psi$ is different from zero.
\end{proof}
\section{Proof of the main theorem}\label{sectionproofmaintheorem}
\begin{proof}[Proof of Theorem \ref{maintheorem}]
    By Lemma \ref{propAB} we separate the argument into four cases. 
    
    In \cite{anarella}, the authors proved that any totally geodesic submanifold  is of type I and is congruent to one of the first three examples given in Theorem \ref{maintheorem}.  
    By Proposition \ref{typeIcurvature}, Proposition \ref{psl} and Proposition \ref{tori}, any non-totally geodesic extrinsically homogeneous Lagrangian submanifold of type I is congruent to either Example \ref{toriexample} or to Example \ref{pslexample}. 

    Proposition \ref{type2asubmanifold} implies that any extrinsically homogeneous Lagrangian submanifold of type II is congruent to an open subset of either the image of the immersion in  Example \ref{type2asubmanifoldexample} or the image of the immersion in Example \ref{type2bsubexample}. 

    Proposition \ref{type3prop} states that any extrinsically homogeneous Lagrangian submanifold of type III is congruent to the one given in Example \ref{type3submanifold}.

    Proposition \ref{propnolagrtypeiv} shows that there are no extrinsically homogeneous Lagrangian submanifolds of type IV.

    Except for the first, sixth and seventh examples in Theorem \ref{maintheorem}, all the submanifolds are not isometric and therefore not congruent. 
    Hence, it only remains to distinguish between the aforementioned cases. 
 
    First, the first submanifold in Theorem \ref{maintheorem} is the first example of Theorem \ref{totgeo}, therefore the only one of these three that is totally geodesic. 
    The sixth submanifold is the immersion $\imath$ given in Example \ref{type2asubmanifoldexample} and the seventh one is the family of immersions $f_\lambda$ given in Example \ref{type2bsubexample}. 

    Suppose that $\iota$ is congruent to $f_\lambda$ for some $\lambda$. 
    That means, there exists a isometry $\mathcal{F}$ of $\Sl\times\Sl$ that maps one into the other. 
    Suppose that $\mathcal{F}\in\Sl\times\Sl\times\Sl\rtimes\Z_2$. 
    These isometries preserve $P$ and $J$, hence $A$ and $B$ have the same shape with respect to $\{E_i\}_i$ and with respect to $\{\mathcal{F}_*E_i\}_i$. 
    Hence, 
    \[
    -\frac{\sqrt{2}}{3}=g(h(E_2,E_2),JE_3)=g(\mathcal{F}_*h(E_2,E_2),\mathcal{F}_*JE_3)=g(h(\mathcal{F}_*E_2,\mathcal{F}_*E_2),J\mathcal{F}_*E_3)=\frac{2\sqrt{2}}{3},
    \]
    which is a contradiction.

    In Theorem \ref{groupofisometries1} we showed that the isometry group of $\Sl\times\Sl$ is a semidirect product of $S_3$ with $\Sl\times\Sl\times\Sl\rtimes\Z_2$. 
    Therefore, to complete the proof we can assume that $\mathcal{F}\in S_3$, i.e. $\mathcal{F}=\Psi_{\kappa,\tau}$ for some $\kappa\in\{0,1\}$, $\tau\in\{0,\tfrac{2\pi}{3},\tfrac{4\pi}{3}\}$. 
    Moreover, we can assume that $\tau\neq0$, since otherwise $\mathcal{F}$ preserves $P$, and therefore we may use the same argument as before, up to sign.  
    From Lemma \ref{isometrieswithAB}it follows that $P$ restricted to $\mathcal{F}(M)$ takes the shape $\tilde A+J \tilde B$, where
    \begin{equation*}
        \begin{split}
             \tilde{A}&=\cos \tau A+(-1)^\kappa \sin \tau B,\\
             \tilde{B}&=-\sin\tau A+(-1)^\kappa \cos \tau B.\\
        \end{split}
    \end{equation*}
For $\tau\neq0$ these matrices have a different form than $A$ and $B$, as it can be seen in Lemma 4 of~\cite{anarella}. 
Therefore, there does not exist such a $\Delta_2$-orthonormal frame $\{E_i\}_i$ such that $A$ and $B$ take type~II form in Lemma \ref{propAB} on $\mathcal{F}(M)$, which is a contradiction.

Similar arguments can be used to distinguish between $f_{\lambda_1}$ and $f_{\lambda_2}$ for $\lambda_1\neq\lambda_2$, by considering the function $\omega_{33}^2=\sqrt{\tfrac{2}{3}}(1-\lambda)$ instead of $h_{22}^3$.
\end{proof}
\section*{Acknowledgements}
The author would like to thank Joeri Van der Veken and Luc Vrancken for their guidance and support throughout the writing of this article. 
\bibliographystyle{abbrv}
\bibliography{homogeneousbib}

\end{document}